\documentclass[10pt,a4paper,twoside]{article}
\usepackage{amsmath}
\usepackage{amsthm}
\usepackage{amssymb,latexsym}
\usepackage{amscd}
\input xy \xyoption{all} \CompileMatrices

\usepackage{ifpdf,paralist}
\ifpdf \usepackage[pdftex,hyperindex]{hyperref}
\else  \usepackage[hypertex,hyperindex]{hyperref}
\fi
\hypersetup{colorlinks=true}

\theoremstyle{plain}
  \newtheorem{theorem}                    {Theorem}       [section]
  \newtheorem{lemma}      [theorem]       {Lemma}
  \newtheorem{definition} [theorem]       {Definition}
  
  \newtheorem{corollary}  [theorem]       {Corollary}
  \newtheorem{proposition}[theorem]       {Proposition}
  
\theoremstyle{definition}
  \newtheorem{remark}     [theorem]       {Remark}
  
  \newtheorem{example}    [theorem]       {Example}

\def\YEAR{\year}\newcount\VOL\VOL=\YEAR\advance\VOL by-1995
\def\firstpage{1}\def\lastpage{1000}
\def\received{}\def\revised{}
\def\communicated{}

\makeatletter
\def\magnification{\afterassignment\m@g\count@}
\def\m@g{\mag=\count@\hsize6.5truein\vsize8.9truein\dimen\footins8truein}
\makeatother

\font\eightrm=cmr8
\font\caps=cmcsc10                    
\font\Caps=cmcsc10 scaled \magstep1   
\font\scaps=cmcsc8

\makeatletter
\@specialpagefalse\headheight=8.5pt
\def\DocMath{{\def\th{\thinspace}\scaps Documenta Math.}}
\renewcommand{\@oddfoot}{\hfill\scaps Documenta Mathematica 
    \number\VOL\  (\number\YEAR) \number\firstpage--\lastpage\hfill}
\renewcommand{\@evenfoot}{\ifnum\thepage>\lastpage\hfill\scaps
    Documenta Mathematica \number\VOL\  (\number\YEAR)\hfill\else\@oddfoot\fi}%

\renewcommand{\@evenhead}{%
    \ifnum\thepage>\lastpage\rlap{\thepage}\hfill%
    \else\rlap{\thepage}\slshape\leftmark\hfill{\caps\SAuthor}\hfill\fi}%
\renewcommand{\@oddhead}{%
    \ifnum\thepage=\firstpage{\DocMath\hfill\llap{\thepage}}%
    \else{\slshape\rightmark}\hfill{\caps\STitle}\hfill\llap{\thepage}\fi}%
\makeatother

\def\TSkip{\bigskip}
\newbox\TheTitle{\obeylines\gdef\GetTitle #1
\ShortTitle  #2
\SubTitle    #3
\Author      #4
\ShortAuthor #5
\EndTitle
{\setbox\TheTitle=\vbox{\baselineskip=20pt\let\par=\cr\obeylines%
\halign{\centerline{\Caps##}\cr\noalign{\medskip}\cr#1\cr}}%
	\copy\TheTitle\TSkip\TSkip%
\def\next{#2}\ifx\next\empty\gdef\STitle{#1}\else\gdef\STitle{#2}\fi%
\def\next{#3}\ifx\next\empty%
    \else\setbox\TheTitle=\vbox{\baselineskip=20pt\let\par=\cr\obeylines%
    \halign{\centerline{\caps##} #3\cr}}\copy\TheTitle\TSkip\TSkip\fi%
\centerline{\caps #4}\TSkip\TSkip%
\def\next{#5}\ifx\next\empty\gdef\SAuthor{#4}\else\gdef\SAuthor{#5}\fi%
\ifx\received\empty\relax
    \else\centerline{\eightrm Received: \received}\fi%
\ifx\revised\empty\TSkip%
    \else\centerline{\eightrm Revised: \revised}\TSkip\fi%
\ifx\communicated\empty\relax
    \else\centerline{\eightrm Communicated by \communicated}\fi\TSkip\TSkip%
\catcode'015=5}}\def\Title{\obeylines\GetTitle}
\def\Abstract{\begingroup\narrower
    \parskip=\medskipamount\parindent=0pt{\caps Abstract. }}
\def\EndAbstract{\par\endgroup\TSkip}

\long\def\MSC#1\EndMSC{\def\arg{#1}\ifx\arg\empty\relax\else
     {\par\narrower\noindent%
     2010 Mathematics Subject Classification: #1\par}\fi}

\long\def\KEY#1\EndKEY{\def\arg{#1}\ifx\arg\empty\relax\else
	{\par\narrower\noindent Keywords and Phrases: #1\par}\fi\TSkip}

\newbox\TheAdd\def\Addresses{\vfill\copy\TheAdd\vfill
    \ifodd\number\lastpage\vfill\eject\phantom{.}\vfill\eject\fi}
{\obeylines\gdef\GetAddress #1
\Address #2 
\Address #3
\Address #4
\EndAddress
{\def\xs{4.3truecm}\parindent=0pt
\setbox0=\vtop{{\obeylines\hsize=\xs#1\par}}\def\next{#2}
\ifx\next\empty 
     \setbox\TheAdd=\hbox to\hsize{\hfill\copy0\hfill}
\else\setbox1=\vtop{{\obeylines\hsize=\xs#2\par}}\def\next{#3}
\ifx\next\empty 
     \setbox\TheAdd=\hbox to\hsize{\hfill\copy0\hfill\copy1\hfill}
\else\setbox2=\vtop{{\obeylines\hsize=\xs#3\par}}\def\next{#4}
\ifx\next\empty\ 
     \setbox\TheAdd=\vtop{\hbox to\hsize{\hfill\copy0\hfill\copy1\hfill}
                \vskip20pt\hbox to\hsize{\hfill\copy2\hfill}}
\else\setbox3=\vtop{{\obeylines\hsize=\xs#4\par}}
     \setbox\TheAdd=\vtop{\hbox to\hsize{\hfill\copy0\hfill\copy1\hfill}
	        \vskip20pt\hbox to\hsize{\hfill\copy2\hfill\copy3\hfill}}
\fi\fi\fi\catcode'015=5}}\gdef\Address{\obeylines\GetAddress}

\def\LOCAL{\jobname.files}

\begin{document}
\Title
Computations in intersection rings of flag bundles
\ShortTitle 
\SubTitle   
\Author 
Daniel R. Grayson\footnote{Grayson has been supported by NSF grants DMS 08-10948, 10-02171.},
Alexandra Seceleanu\footnote{Seceleanu has been supported by NSF grant DMS 11-03176, 21-01225.},
and 
Michael E. Stillman\footnote{Stillman has been supported by NSF grants DMS 08-10909, 10-02210, 20-01367.}
\ShortAuthor 
Grayson, Seceleanu, and Stillman
\EndTitle

\def\degreemult#1#2{{#1}\langle#2\rangle}
\def\Sq{\operatorname{Sq}}
\def\Spec{\operatorname{Spec}}
\def\Mult{\operatorname{Mult}}
\def\Pr{\operatorname{Pr}}
\def\l{\operatorname{\ell}}
\def\lm{\operatorname{\it \ell m}}
\def\lc{\operatorname{\it \ell c}}
\def\implies{\Rightarrow}
\def\iff{\Leftrightarrow}
\def\impliedby{\Leftarrow}
\def\deg{\operatorname{deg}}
\def\top{{\rm top}}
\def\PF{\operatorname{PF}}
\def\rank{\operatorname{rank}}
\def\rem{\,\%\,}
\def\quo{\,{/}\!{/}\,}
\def\iso{{\rm iso}}
\def\isom{\cong}
\def\isomap{\stackrel\isom\to}
\def\ph{\varphi}
\def\tph{\tilde\varphi}
\def\red{\operatorname{red}}
\def\Fl{\operatorname{Fl}}
\def\FF{{\mathbb F}}
\def\calF{{\mathcal F}}
\def\QQ{{\mathbb Q}}
\def\PP{{\mathbb P}}
\def\ZZ{{\mathbb Z}}
\def\NN{{\mathbb N}}
\def\LL{{\mathcal L}}
\def\OO{{\mathcal O}}
\def\EE{{\mathcal E}}
\def\DD{{\mathcal D}}
\def\gb{\operatorname{gb}}
\def\Gb{Gr\"obner basis}
\def\Gbs{Gr\"obner bases}
\def\Split{\operatorname{Split}}

\Abstract 
Let $\EE$ be a vector bundle on a nonsingular variety $X$.  A partition of its
rank as an ordered sum of natural numbers gives rise to a flag variety $\FF$
whose points parametrize filtrations of $\EE$ whose consecutive subquotients
have ranks given by the members of the partition.  The Chern classes of the
corresponding tautological bundles on $\FF$ generate the intersection ring
$A(\FF)$ of $\FF$.  Using a convenient ordering of the monomials in the Chern
classes, we give an explicit description of a minimal Gr\"obner basis for the
corresponding ideal of relations presenting $A(\FF)$.

This is used it to justify an algorithm implemented in {\em Macaulay2} for
computation of the push-forward of a cycle class.

We also generalize the result and algorithm to cover isotropic flag bundles.

\EndAbstract
\MSC 
14C17
\EndMSC
\KEY 
flag bundle,
isotropic flag bundle,
intersection ring,
push-forward map,
Gr\"obner basis,
Chern classes
\EndKEY
\Address 
2813 Belmont Drive
Henderson
Nevada 89074
USA
{\href{http://dangrayson.com}{http://dangrayson.com}}
danielrichardgrayson\char`\@gmail.com
\Address
Department of Mathematics
University of Nebraska
Lincoln, Nebraska 68588
USA
{\href{http://www.math.unl.edu/~aseceleanu2/}{http://www.math.unl.edu/\char`\~aseceleanu2/}}
aseceleanu\char`\@unl.edu
\Address
Department of Mathematics
Cornell University
Ithaca, New York 14853
USA
{\href{http://math.cornell.edu/~mike/}{http://math.cornell.edu/\char`\~mike/}}
mike\char`\@math.cornell.edu
\Address
\EndAddress
\def\degreemult#1#2{{#1}\langle#2\rangle}
\def\Sq{\operatorname{Sq}}
\def\Spec{\operatorname{Spec}}
\def\Mult{\operatorname{Mult}}
\def\Pr{\operatorname{Pr}}
\def\l{\operatorname{\ell}}
\def\lm{\operatorname{\it \ell m}}
\def\lc{\operatorname{\it \ell c}}
\def\implies{\Rightarrow}
\def\iff{\Leftrightarrow}
\def\impliedby{\Leftarrow}
\def\deg{\operatorname{deg}}
\def\top{{\rm top}}
\def\PF{\operatorname{PF}}
\def\rank{\operatorname{rank}}
\def\rem{\,\%\,}
\def\quo{\,{/}\!{/}\,}
\def\iso{{\rm iso}}
\def\isom{\cong}
\def\isomap{\stackrel\isom\to}
\def\ph{\varphi}
\def\tph{\tilde\varphi}
\def\red{\operatorname{red}}
\def\Fl{\operatorname{Fl}}
\def\FF{{\mathbb F}}
\def\calF{{\mathcal F}}
\def\QQ{{\mathbb Q}}
\def\PP{{\mathbb P}}
\def\ZZ{{\mathbb Z}}
\def\NN{{\mathbb N}}
\def\LL{{\mathcal L}}
\def\OO{{\mathcal O}}
\def\EE{{\mathcal E}}
\def\DD{{\mathcal D}}
\def\gb{\operatorname{gb}}
\def\Gb{Gr\"obner basis}
\def\Gbs{Gr\"obner bases}
\def\Split{\operatorname{Split}}

\section*{Introduction}

Let $A$ be a commutative ring and let $f(x) = x^n + b_1 x^{n-1} + \dots + b_n$
be a monic polynomial of degree $n$ with coefficients $b_i \in A$.  Let $n =
n_1 + \dots + n_r$ be an expression of $n$ as a sum of $r$ natural numbers. Consider the polynomial ring
$R=A[c_{1,1},\dots,c_{1,n_1};\dots;c_{r,1},\dots,c_{r,n_r}]$ and for $i = 1,
\dots, r$ the monic polynomials
\[g_i(x) = x^{n_i} + c_{i,1} x^{n_i-1} + \dots + c_{i,n_i} \in
R[x].
\]  Let $I \subseteq R$ be the ideal generated by the polynomial equations arising from identifying the coefficients of the following identity in $R[x]$.
\[f(x) = g_1(x) \cdot {\dots} \cdot g_r(x).\]  The ring
$\PF_{n_1,\dots,n_r}(f) := R/I$ is the universal $A$-algebra supporting a
factorization of $f(x)$ into a sequence of $r$ monic factors  having degrees ${n_1,\dots,n_r}$.  If $A$ is a graded ring and $\deg b_i =i$ for all $i$, then we may set $\deg c_{i,j} = j$ for all $i$ and $j$,
ensuring that the  ideal $I$ and the quotient ring $R/I$ are 
homogeneous.

In this paper we examine Gr\"obner bases of the family of ideals described above.  We show that, for a
suitable ordering of the monomials in $R$, the lead terms of a \Gb{} of the
ideal $I$ can be explicitly described.  The monomials not divisible by such a
lead term give an explicit finite basis for $\PF_{n_1,\dots,n_r}(f)$.

Our motivation and interest come from intersection theory.  Let $\EE$ be a
vector bundle of rank $n$ on a nonsingular variety $X$, and let $c(\EE) = 1 +
c_1\EE + \dots + c_n \EE$ denote the total Chern class (or Chern polynomial) of
$\EE$ in the (graded) intersection (Chow) ring $A(X)$ of algebraic cycles on $X$
modulo rational equivalence.  Let $\pi : \FF := \Fl_{n_1,\dots,n_r}(\EE) \to X$
denote the flag bundle of $\EE$ that parametrizes filtrations of $\EE$ over points of $X$, where the successive
subquotients from the filtration have ranks $n_1, \dots, n_r$.  In other words, $\FF$ supports a
{\em tautological} filtration 
\[\pi^* \EE = \EE_r \supseteq \EE_{r-1} \supseteq \dots
\supseteq \EE_1 \supseteq \EE_0 = 0\]
 with $\rank(\EE_i/\EE_{i-1}) = n_i$ for $i
= 1, \dots, r$, and with $\FF$ being the universal variety over $X$ with such a
filtration.  The Chern polynomial of $\EE$, pulled back to $\FF$, acquires
factors 
\[\pi^* c(\EE) = c(\pi^*\EE) = c(\EE_r/\EE_{r-1}) \cdot \dots \cdot
c(\EE_2/\EE_1) \cdot c(\EE_1/\EE_0)\] whose degrees are $n_1,\dots,n_r$.
Grothendieck showed (see \cite[Theorem 1, p.~4-19]{grothendieck-1958} and
Theorem \ref{groththm} below) that the corresponding homomorphism
$\PF_{n_1,\dots,n_r}(c_x(\EE)) \to A(\FF)$ is an isomorphism, where $c_x(\EE)$
denotes the polynomial $x^n + c_1(\EE) x^{n-1} + \dots + c_n(\EE) \in A(X)[x]$.
Our Theorem \ref{theorem:ourtheorem} provides an explicit basis of $A(\FF)$ as a free
$A(X)$-module, and it follows from Corollary \ref{computepilowerstar} that the
map $\pi_* : A(\FF) \to A(X)$ can be computed by taking the coefficient of the
basis element of highest degree.

In section \ref{section:isotropic-flag-bundle} we generalize the results
referred to above to the case of {\em isotropic} Grassmannians and flag
bundles. 

The identification of the leading terms of the ideals defining intersection
rings of flag bundles and isotropic flag bundles in this paper was suggested
and motivated by computations with {\em Macaulay2} \cite{M2}.  The resulting
algorithm for computing push-forward maps on intersection rings of flag
bundles, described and justified by Corollary \ref{computepilowerstar}, was
implemented (for Grassmannian bundles) in the {\em Maple} package {\em
  Schubert} \cite{schubert}, written by Sheldon Katz and Stein-Arild Str\o{}mme
starting in November, 1991, and (for flag bundles) in its intended replacement,
the {\em Macaulay2} package {\em Schubert2}, written by Grayson, Stillman,
Str\o{}mme,  Eisenbud, and  Crissman, starting in May, 2007. This paper provides a rigorous treatment of the theory supporting these computations.

Further work on algebras defining universal factorizations appears in a recent preprint of Sam--Snowden \cite{SS21} where properties of these rings are studied and used to compute the cohomology of flag supermanifolds.

\section{\Gbs{} over arbitrary coefficient rings}\label{section:arb}

Let $A$ be a commutative ring.  We'll assume nothing is known, algorithmically, about $A$ other than the basic ring operations:
addition, subtraction, multiplication, and equality.  In particular, there are no algorithms for testing divisibility.

Given $m \in \NN$, let $[x] := [x_1, \dots, x_m]$ denote the free commutative monoid consisting of the monomials $x^\alpha =
x_1^{\alpha_1} \dots x_m^{\alpha_m}$ formed from the generators $x_1, \dots, x_m$.  
A {\em monomial ordering} on $[x]$ is an ordering on $[x]$ that is compatible
with multiplication and makes $[x]$ well ordered. Compatibility of $<$ with multiplication means that $u<v$ implies $mu<mv$ for any monomials $m,u,v$, and the property that $[x]$ is well ordered by $<$  is equivalent to requiring  that $1<x^\alpha$ for any monomial $x^\alpha\neq 1$. The latter property is sometimes referred to as $<$ being a global ordering. All monomial orderings used in this paper are global orderings. 

Throughout this section, let $R$ be the polynomial ring (and monoid ring) $R=A[x] = A[x_1, \dots, x_m]$.

Let $<$ be a monomial ordering of $[x]$.
For a nonzero polynomial $f \in R$ we write $\l_<(f) = \lc_<(f) \lm_<(f)$, where
$\l_<(f) \in R$ is the lead term of $f$, $\lc_<(f) \in A$ is the lead coefficient
of $f$, and $\lm_<(f) \in [x]$ is the lead monomial of $f$.
A polynomial $f \in R$ will be called {\em monic} (with respect to $<$) if $f \ne 0$ and $\lc_<(f) =
1$.  An ideal $I$ of $R$ will be called {\em monic} (with respect to $<$) if the ideal $\l_<(I)$
generated by its lead terms can be generated by a set of (monic) monomials.  Thus an ideal $I$ is monic
if and only if for any nonzero $a \in A$, if $a x^\alpha \in \l_<(I)$ then
$x^\alpha \in \l_<(I)$.
We say that $R/I$ is a {\em monic} $A$-algebra (with respect to $<$) when $I$ is monic.

Let $<$ be a monomial ordering of $[x]$ and $I$ a monic ideal. Intersection of $\l_<(I)$
with the Noetherian ring $\ZZ[x_1, \dots, x_m]$ shows that $\l_<(I)$ can be
generated by a finite number of monomials.
A finite subset $G$ of $I$ is called a {\em \Gb{} of $I$} (with respect to $<$) if $\l_<(G) := \{\l_<(g)\mid g \in G\}$ generates $\l_<(I)$. Because $[x]$ is well ordered, it follows
that such a subset $G$ generates $I$. A subset $G$ of $R$ is called a {\em \Gb} (with respect to $<$) if it is a \Gb{}
of the ideal it generates.  A {\em monic \Gb} (with respect to $<$) is a \Gb{} that consists of monic polynomials.  An ideal $I$ is monic (with respect to $<$) if and only if it has a monic \Gb.  Say
that a monic \Gb{} is {\em minimal} if no lead term of an element divides the
lead term of another element.  Any monic \Gb{} can be converted to a minimal
one by discarding each element whose lead term is divisible by the lead term of
another element.  A \Gb{} $G$ is called {\em auto-reduced} if it is monic and no term in any polynomial $g_i\in G$ is divisible by $\l_<(g_j)$, whenever $i\neq j$.

{\em Reduction} of a polynomial $f \in R$ modulo a set $G$ of monic polynomials
is the nondeterministic algorithm that proceeds as follows: replace any
monomial $x^{\alpha+\beta}$ occurring in $f$ by $x^\alpha h$ if there is an
element of $G$ of the form $x^\beta - h$ with lead term $x^\beta$, repeating as
long as possible.  The algorithm terminates because $[x]$ is well ordered.  

If
$G$ is a monic \Gb{} of $I$, then the reduction algorithm modulo $G$ yields a result for which no monomial is divisible by a lead monomial of $I$.  Moreover, the result is unchanged if $f$ is replaced by any element of the coset of $f$ in $R/I$.  In particular, the reduction is $0$ if and only if $f \in I$. 
The reduction for a polynomial $f \in R$ modulo a \Gb{} for the ideal $I$  does not depend on the choice of \Gb{} for $I$, but only on the choice of monomial ordering (a proof for this fact follows along the lines of \cite[Chapter 2, \S 6, Exercise 4]{CLO}). Thus the following map is well defined.

\begin{definition}\label{definition:redI}
  Let $<$ be a monomial ordering of $[x]$.
  For a monic ideal $I\subset R$, we let $\red_I : R \to R$ denote the $A$-linear
  function that sends $f$ to its reduction modulo a \Gb{} of $I$. 
\end{definition}

We recall now a few  useful equivalent characterizations for  \Gbs{}. 

\begin{lemma}\label{prop:equivGB}
  Let $<$ be a monomial ordering of $[x]$.
 Let $G$ be a set of monic polynomials that generates an ideal $I\subset R$. The following statements are equivalent:
 \begin{enumerate}
 \item $G$ is a monic \Gb{} of $I$ with respect to $<$;
 \item if $f\in I$, then there exists $g\in G$ so that $\l_<(f)$ is divisible by $\l_<(g)$; and
 \item every $f\in I$ reduces to 0 with respect to $G$ and $<$.
 \end{enumerate}
 Furthermore, each of the statements above implies that 
 \begin{enumerate}
 \item[4.] the function $\red_I$ induces an $A$-linear isomorphism $R/I \isomap \red_I(R)$ and $\red_I(R)$ is a free $A$-submodule of $R$ generated by the monomials not contained in $\l_<(I)$.
 \end{enumerate}
\end{lemma}
\begin{proof}
The case $R=K[x]$, with $K$ a field is addressed in \cite[Proposition 5.38]{BW}. In particular, statements 2 and 3  above appear as statements (iii) and (i) in the referenced result and are shown to be equivalent characterizations for $G$ being a \Gb{} for $I$. The same arguments apply verbatim if the ground field $K$ is replaced by a commutative ring $A$ (and thus $R=A[x]$) provided that division by elements of $G$ is possible. This is the case since the elements in $G$ are assumed to be monic, hence the equivalence of the first three statements follows. Furthermore, assuming that $G$ is a monic \Gb{} of $I$ yields that $I$ is monic, so the function $\red_I$ is well defined. Moreover, by part 3, $I$ is contained in the kernel of this map and the converse containment is clear from the description of the reduction process. The validity of statement 4 thus follows from the first isomorphism theorem.
\end{proof}

 We record below some aspects of the behavior of Gr\"obner bases under base change which are important for computations.

\begin{lemma}\label{lemma:basechange}
  A monic \Gb{} $G$ of $I \subseteq A[x]$ with respect to a monomial ordering $<$ of $[x]$ is stable under base change along a ring homomorphism $\ph : A \to A'$ in the sense that $\ph(G) \subseteq A'[x]$
  is a \Gb{} of $I\cdot A'[x]$ (with respect to $<$).
\end{lemma}

\begin{proof}
  The basis of the free $A$-module $A[x]/I$ given by the monomials not
  contained in $\l_<(I)$ yields a basis of
  $(A[x]/I) \otimes_A A' \cong A'[x]/(I\cdot A'[x])$
  over $A'$ consisting of the monomials in $\ph(\l_<(G))$.
  Thus none of those monomials can occur as leading terms in any element of
  $I\cdot A'[x]$. The equivalence of statements 1 and 3 in Lemma
  \ref{prop:equivGB} yields the result.
\end{proof}

\begin{lemma} \label{lemma:ordering}
Let $I \subseteq A[x]$ be a monic ideal with respect to a monomial ordering $<$, and let $G$ be a monic \Gb{} of $I$ with respect to $<$.  Suppose $\ll$ is also a monomial ordering on $[x]$, and suppose that each $g \in G$ has the same lead term for $<$ and $\ll$.  Then $G$  is
  a \Gb{} of $I$ with respect to $\ll$.
\end{lemma}
\begin{proof}
  According to the implication $1 \Leftarrow 2$ in Lemma \ref{prop:equivGB},
  it will be enough to show that
  given $f \in I$ with $f \ne 0$ and $\l_\ll(f)=a x^\alpha$,
  the monomial $x^\alpha$ is divisible by a lead monomial of $G$. Reduction of $f$ by $G$ with respect to $<$ succeeds in reducing $f$ to $0$ because $G$ is a \Gb{}.  The same is true upon reduction of $f$ by $G$ with respect to $\ll$ because the reduction algorithm depends
  only on the set $G$ and on which terms of its elements are the lead terms,
  and those have not changed; i.e., the same sequence of division steps can be used.
  During the successful reduction, individual terms of $f$ are replaced by terms that are smaller in the ordering $\ll$.
  Hence the lead term $\l_\ll(f)=a x^\alpha$ must eventually be one of the terms that is reduced by the divison process,
  rendering $x^\alpha$ divisible by a lead monomial of $G$.
\end{proof}

\begin{lemma}\label{lemma:gbfields}
Fix a monomial ordering $<$  on $[x]$. A finite set of monic polynomials  $G$ is a \Gb{} in $ \ZZ[x] = \ZZ[x_1, \dots, x_m]$ if and only if it is a \Gb{} in $ \QQ[x] = \QQ[x_1, \dots, x_m]$. 
\end{lemma}

\begin{proof}
Let $I$ be the ideal generated by $G$ in $\ZZ[x]$ and $J$ the ideal generated by $G$ in $\QQ[x]$.  By Lemma  \ref{prop:equivGB}, the statement reduces to saying that for each $f\in I$ there exists a $g\in G$ so that $\l_<(f)$ is divisible by $\l_<(g)$ if and only if for each $f\in J$ there exists a $g\in G$ so that  $\l_<(f)$ is divisible by $\l_<(g)$. Since  for each $f\in J$ there is an integer $c$ such that $cf\in I$,  and since $\l_<(g)$ is monic for all $g\in G$, the desired equivalence follows.  \end{proof}

\begin{remark}
A statement analogous to Lemma \ref{lemma:gbfields} holds for any field $F$:  a finite set of monic polynomials  $G$ is a \Gb{} in $ \ZZ[x] = \ZZ[x_1, \dots, x_m]$ if and only if the image of $G$ under the map $\pi':\ZZ[x]\to F[x] =F[x_1, \dots, x_m]$  induced by the unique ring homomorphism $\pi:\ZZ\to F$ is a \Gb{} in $F[x]$. In particular this holds if $F$ is a field of positive characteristic. The proof relies on the fact that $\pi'$ preserves the leading terms of polynomials in $G$ since these are monic. 
\end{remark}
  
We now consider towers of polynomial rings of the form $A[t_1,\dots,t_n][x_1,\dots,x_m]$ with a focus on several monomial orderings that can arise and their interplay. A polynomial in such a ring is a formal linear combination of monomials in the $x$ variables, whereas a polynomial in $A[t_1,\dots,t_n,x_1,\dots,x_m]$ is a formal linear combination of monomials in all the variables. It is important to distinguish the two situations because we will be  ordering the monomials in these rings.
    
  \begin{definition}\label{prodorddef}
  Suppose $[t] := [t_1,\dots,t_n]$ and $[x] := [x_1,\dots,x_m]$ are monoids
  with monomial orderings $<_t$ and $<_x$ respectively.  The {\em block ordering} denoted by $<_{x,t}$ on the monoid
  $[t,x] := [t_1,\dots,t_n, x_1,\dots,x_m]$ is defined by 
  \[x^\alpha t^\beta <_{x,t} x^\gamma t^\delta \text{\quad  if and only if \quad either }x^\alpha <_x
  x^\gamma, \text{ or } x^\alpha = x^\gamma \text{and }t^\beta <_t t^\delta.
  \] We will often employ the shorthand notation $[x]\gg[t]$ for the block ordering $<_{x,t}$ except when explicit comparisons between individual polynomials are necessary. 
 
 If $[x_1],[x_2],\ldots[x_n]$ are monoids one defines the block ordering $[x_1]\gg[x_2]\gg\ldots\gg[x_n]$ by induction on $n$, as the ordering $([x_1]\gg[x_2])\gg\ldots)\gg[x_n]$.
\end{definition}

When considering the monomial orderings $<_x, <_t$ and $<_{x,t}$ defined above, $\l_x, \l_t $ and $\l_{x,t}$ denote the lead terms for the respective orderings. Similar notation is used for the lead monomial and the lead coefficient. The embedding $A[t] \hookrightarrow A[t,x]$ preserves the monomial ordering and thus a
monic polynomial of $A[t]$ remains monic and has the same lead monomial when
regarded as a polynomial in $A[t,x]$. Let $\psi : A[t][x] \to A[t,x]$ be the isomorphism relating the tower ring
$A[t][x]$ to its {\em flattened} form $A[t,x]$. For $h \in A[t,x]$, we note
that the definition of the block ordering gives $\l_{x,t}(\psi(h)) = \l_{t}(\lc_{x}(h)) \lm_{x}(h)$.

\begin{definition}
\label{<<to<x}
Given an ordering $\ll$ on the monomials of $A[t,x]$, which need not be the block ordering described above, one may speak of the monomial ordering $<_x$ on $[x]$ (equivalently on $A[t][x]$) induced from $\ll$ by 
setting  $m<_x m'$ in $[x]$ whenever $m\ll m'$ in $[t,x]$.  
\end{definition}

\begin{proposition} \label{lemma:ordering-b}
  Suppose that $I \subseteq A[t,x] = A[t_1,\dots,t_n,x_1,\dots,x_m]$ is a monic ideal
  and that $G$ is a monic \Gb{} of $I$ with respect to an ordering $\ll$, such that the lead terms of $G$ involve none of the variables $t_i$.  
  
  Consider the monomial ordering $<_x$ on $A[t][x]$ obtained by restricting the  ordering $\ll$ to $[x]$ as in definition \ref{<<to<x} and let $\psi : A[t][x] \to A[t,x]$ be the isomorphism relating the tower ring $A[t][x]$ to its flattened form $A[t,x]$. Then the set $\psi^{-1}(G)$ is a \Gb{} for $\Psi^{-1}(I)$ with respect to the ordering $<_x$.
\end{proposition}
\begin{proof}
First note that the elements of $\Psi^{-1}(G)$ retain the same lead terms with respect to $<_x$ as for $\ll$  because $x^\alpha \gg t^\beta x^\gamma\gg x^\gamma$ whenever $x^\alpha$ is the leading term of a polynomial in $G$ and  $t^\beta x^\gamma$ is another monomial in the same polynomial. (Here we use that $t^\beta \gg 1$ for any monomial $t^\beta\neq 1$.) Now Lemma  \ref{lemma:ordering} can be applied to deduce that $\Psi^{-1}(G)$ is a monic \Gb{} of $\Psi^{-1}(I)$ with respect to the monomial ordering $<_x$. 

\end{proof}

\begin{lemma}\label{lemma:specialization}
  Let $R = \ZZ[t,x] = \ZZ[t_1, \dots, t_m,x_1,\dots,x_n]$ be a polynomial ring
  with a monomial ordering $\ll$.  Let $G\subseteq R$ be a monic \Gb{} whose lead
  monomials involve none of the elements $t_i$.  Let $A$ be a commutative ring and let $\ph : R \to
  A[x]$ be a homomorphism such that $\ph(t_i) \in A$ for all $i$ and $\ph(x_j)
  = x_j$ for all $j$.  Equip $A[x]$ with the monomial ordering $<_x$ induced from
 $\ll$ as in Definition \ref{<<to<x}.  Then $\ph(G)$ is a monic \Gb{} with respect to $<_x$.
\end{lemma}
\begin{proof}
  In the case where $A=\ZZ[t]$ and $\ph:\ZZ[t,x] \to \ZZ[t][x]$ is the inverse of the map $\Psi$ in Proposition \ref{lemma:ordering-b}, the result is given by Proposition \ref{lemma:ordering-b}.  The general case follows because monic \Gbs{} are stable under base change, by Lemma \ref{lemma:basechange}.
\end{proof}

\subsection*{Composition of \Gbs}\label{section:composition}

The next result shows how to compose a pair of \Gbs{} in a tower of polynomial rings $A[t,x]$ equipped with a monomial ordering $\ll$ on $[t,x]$ and an induced monomial ordering $<_x$ given by Definition \ref{<<to<x}.

Let $I \subseteq A[t]$ be an ideal, and let $B = A[t]/I$.  A monic polynomial $g$
of $B[x]$ can be lifted (noncanonically) to a monic polynomial $\tilde g \in
A[t,x]$ by lifting its coefficients to $A[t]$ in such a way that the property of a coefficient of a given monomial in $[x]$ being equal to $1$ or to $0$ is preserved.  In the process, no new monomials in $x$ are
introduced, and $\lm_x(\tilde g) = \lm_x(g)$. If $J$ is an ideal of $B[x]$, let $\tilde J$ denote its preimage in $A[t,x]$. 

\begin{proposition}\label{composite}
Suppose that $B = A[t]/I$ is a monic $A$-algebra, and that $B[x]/J$ is a monic $B$-algebra. 
Let $G' \subseteq A[t]$ and $G'' \subseteq B[x]$ be monic \Gbs{} for $I$ and $J$ respectively.
Let $A[t,x]$ be equipped with the block ordering $<_{x,t}$.
Then
\begin{enumerate}
\item $A[t,x]/\tilde{J}$ is a monic $A$-algebra, and
\item $G := G' \cup \{ \tilde{h} \mid h \in G'' \}$ is a \Gb{} of $\tilde{J}$ with respect to $<_{x,t}$.
  Moreover, $\l_{x,t}(\tilde{J})$ is generated by the monomials appearing as lead monomials of elements of $I$ or of 
  $J$.
\end{enumerate}
\end{proposition}

\begin{proof}
  The set $G$ consists of monic polynomials. Together with assertion 1, this observation is sufficient to validate assertion 2. Let $h$ be a nonzero element of $\tilde{J}$.  Our goal is to show that $\lm_{x,t}(h)$ is
  divisible by a lead term of $G$.  There are two cases.

  If $\lc_{x}(h) \in I$, then $\lm_t(\lc_x(h))$ is divisible by the lead term of an element of $G'$ because $G'$ is a \Gb{} for $I$.
  
  If $\lc_x(h) \notin I$, then $\lm_{x,t}(h) = \lm_t(\bar h)$, where $\bar h$ is the image of $h$ in $J$, and $\lm_t(\bar h)$ is divisible by the lead term of an element of $G''$ because $\bar h \in J$.

  In both cases, $\lm_{x,t}(g) = \lm_t(\lc_x(h)) \lm_x(h)$ is divisible by the lead term of the corresponding element of $G$.
\end{proof}

\section{Universal factorization}\label{section:unifac}

Now we apply the results of section \ref{section:arb} to study the \Gb{} of the
ideal defining the scheme parametrizing factorizations of a polynomial.

Let $A$ be a commutative ring.  For polynomials $g$ and $h$ in $A[x]$ we let
$\langle g \equiv h \rangle$ denote the ideal of $A$ generated by the
coefficients of $g-h$.  We introduce $f \rem g$ as notation for the remainder
and $f \quo g$ as notation for the quotient when $f$ is divided by a monic
polynomial $g$.

\begin{definition}\label{PFdef}
  Let $f(x) \in A[x]$ be a monic polynomial of degree $n$.  Given $r \in \NN$
  and $n_1, \dots, n_r \in \NN$, fix the polynomial ring 
  \[R=A[c_{ij} \vert 1 \leq i \leq r,1 \leq j \leq n_i],\]
   let $g_i = x^{n_i} + c_{i1}x^{n_i-1} + \cdots + c_{in_i} \in
  R[x]$ for $1 \leq i \leq r$, and  define the ideal $I = \langle f \rem g_1 g_2 \dots
  g_r \equiv 0 \rangle \subseteq R$.  Set  $\PF_{n_1,\dots, n_r}(f) := R/I$.
  
  \noindent In particular, if $r=1$, and $g_1(x)=x^m+c_1x^{m-1}+\ldots+c_m$ then we have
  $$PF_m(f)=A[c_1,\ldots,c_m]/\langle f\%g\equiv 0 \rangle$$
  and, if  $n_1=n_2=\dots=n_r=1$, we write $\PF_{1,\dots, 1}(f)=\PF_{r:1}(f)$. 
 \end{definition}

Definition \ref{PFdef} will usually be applied in the case where $n \geq n_1 + \cdots +n_r$, for otherwise $I$ is the unit ideal and $\PF_{n_1,\dots, n_r}(f)=0$. We now show that the algebra $\PF_{n_1,\dots, n_r}(f)$ is graded with respect to an appropriate assignment of degrees to the variables of $R$ as well as to the coefficients of $f(x)$.

\begin{lemma}\label{lemma:homogeneous}
    If $A$ is graded and if the coefficients of $f(x)=x^n+b_1x^{n-1}+\ldots+b_{n-1}x+b_n$ are homogeneous, with $\deg(b_i)=i$,
then the defining ideal of $\PF_{n_1,\dots, n_r}(f)$, i.e., $I = \langle f \rem g_1 g_2 \dots
  g_r \equiv 0 \rangle \subseteq R$ of Definition \ref{PFdef}, is a homogeneous ideal of $R$, where $R$ is graded by setting $\deg(c_{ij}) = j$.
\end{lemma}
\begin{proof}
  Set $g(x)=g_1(x)g_2(x)\ldots g_r(x)$ and $m=\sum_{i=1}^r n_i$. Notice that
  each $g_i(x)$, and consequently also $g(x)$, is homogeneous under the
  assignment $\deg(c_{ij})=j$ and $\deg(x)=1$. The ideal $I$ is generated by
  the coefficients of $f(x) \rem g(x)$, and this remainder is in turn
  homogeneous of degree $n-m$.  Thus the coefficients of
  $x^{m-1},x^{m-2},\dots,x,1$ in $f(x) \rem g(x)$ are homogeneous and have
  complementary degrees $n-m+1, n-m+2,\dots, n-1, n$.
\end{proof}

\begin{remark}\label{PFfunctor}
  The $A$-algebra $\PF_{n_1,\dots,n_r}(f)$ represents the functor that assigns
  to each $A$-algebra $B$ the set of factorizations of $f$ in $B[x]$ with monic
  factors of degrees $n_1, n_2, \dots, n_r, n-\sum_i n_i$, and thus is the
  universal $A$-algebra supporting a factorization of that form. 
  To see this, let $n_{r+1}=n-\sum_{i=1}^r n_i$, define a larger polynomial ring
  $R'=A[c_{ij}, 1 \leq i \leq r+1,1 \leq j
    \leq n_i]$ and let $g_i = x^{n_i} + c_{i1}x^{n_i-1} + \cdots + c_{in_i} \in
  R'[x]$ for $1 \leq i \leq r+1$.  Consider the ideal $$I' = \langle f \equiv g_1 g_2 \dots
  g_r g_{r+1} \rangle \subseteq R'$$ which encodes the splitting of the polynomial $f(x)$ with monic factors of degree $n_1, n_2,\ldots,n_{r+1}$. Set $\Split_{n_1, n_2,\ldots,n_{r+1}}(f)=R'/I'$. Then the map 
  $$ \Split_{n_1,\dots,n_{r+1}}(f) \to \PF_{n_1,\dots,n_{r-1}}(f)$$ that sends $c_{r+1,i}$ to the coefficient of  $x^{n_{r+1}-i}$ in $f(x)\quo g_1(x)g_2(x)\cdots g_r(x)$, viewed as a polynomial in $R[x]$, is an isomorphism.
 \end{remark}

\begin{remark}
The algebra $\PF_{r:1}(f)$ is isomorphic to the {\em splitting algebra} of $f(x)$ as defined in 
\cite[Section 1.2]{laksov-MR2130463} and \cite[Definition1.2]{laksov-MR2728150}.  When $A$ is a field and $r+1=n$, every residue field of $\PF_{r:1}(f)$  is isomorphic to a splitting field of $f$. Furthermore, in that case, the defining ideal of $\Split_{r+1:1}(f)$ is given by
  $$I'=\langle b_i-s_i(c_{11},\ldots,c_{n,1}), 1\leq i\leq n \rangle,$$ where $s_i$ denotes the $i^{th}$ symmetric polynomial on  variables $c_{11},\ldots,c_{n,1}$. A \Gb{} for this ideal $I'$ is given in \cite[Theorem 1.2.7]{MR1255980}. The main result of this paper generalizes this theorem.
\end{remark}


Our goal is to show that $\PF_{n_1,\dots,n_r}(f)$ is a monic $A$-algebra and
to identify the lead terms of the \Gb{} of the corresponding ideal with
respect to a suitable monomial ordering.  We shall achieve this goal by performing induction on the number of factors, $r$. We begin with the case where $r=1$.

\begin{definition}
\label{defM(g,f)}
Let $f(x), g(x)\in A[x]$ be monic polynomials of degree $n$ and $m$ respectively. For a non-negative integer $r$, let $F_r$ denote the free $A$-submodule of $A[x]$ equipped with basis $\{x^r, x^{r-1}, \dots, x, 1\}$.  Consider the map of free $A$-modules  $$\phi :F_{n-m}\oplus F_0 \to F_n \mbox{ defined by }(p,c) \mapsto gp+fc.$$ Let $M(g,f)$ be the matrix representing this map with respect to the given bases. 
\end{definition}
\noindent From the definition we see  $M(g,f)$ is an $(n+1) \times (n-m+2)$ matrix with entries in $A$. We shall refer to the top $n-m+1$ rows of $M(f,g)$ as the top part of the matrix and to the bottom $m$ rows as the bottom part of the matrix. 

\begin{example}\label{example:matrixform}
  If $n = 5$ and $m = 2$, and 
  \begin{eqnarray*}
  f(x) &=& x^5+b_1x^4+b_2x^3+b_3x^2+b_4x+b_5\\
  g(x) &=& x^2+c_1x+c_2,
  \end{eqnarray*}
  then 
   the matrix $M=M(g,f)$ is given by
  \[M = \begin{pmatrix}
               1 &   0 &   0 & 0   &1  \\
             c_1 &   1 &   0 & 0   &b_1\\
             c_2 & c_1 &   1 & 0   &b_2\\
               0 & c_2 & c_1 & 1   &b_3\\
               0 &   0 & c_2 & c_1 &b_4\\
               0 &   0 &   0 & c_2 &b_5\\
        \end{pmatrix}.
  \]
  Notice the location of the $1$'s on the diagonal in the top part of the matrix. These entries can be used as pivots with row and column operations to obtain an equivalent matrix of the form 
  \[M' = \begin{pmatrix}
               1 &   0 &   0 & 0   &0  \\
             0 &   1 &   0 & 0   &0\\
             0 & 0 &   1 & 0   &0\\
               0 & 0 & 0 & 1   &0\\
               0 &   0 & 0 & 0 &\alpha\\
               0 &   0 &   0 & 0 &\beta \\
        \end{pmatrix}.
  \]
  We prove in Proposition \ref{minors} that one can get from this the identity $f\%g=\alpha x+\beta$.
   \end{example}

\begin{definition}
  For $i = (i_0, i_1, \dots, i_{n-m+1})$ with $n \geq i_0 > i_1 > \dots >
  i_{n-m+1} \geq 0$, let $F_i$ denote the free $A$-submodule of $F_{n}$ with
  basis $\{x^{i_0}, x^{i_1}, \dots, x^{i_{n-m+1}}\}$, and let $\pi_i : F_{n}
  \to F_i$ denote the projection map.  Using the notation from Definition \ref{defM(g,f)}, 
  let $M_i = \pi_i \circ M(g,f)$. Then $M_i$ is
  a square submatrix of $M$, and its determinant is a maximal $(n-m+2)
  \times (n-m+2)$ minor of $M(g,f)$.
\end{definition}

We are now ready to relate the maximal minors of $M(g,f)$ to the coefficients of the remainder polynomial $f\%g$.

\begin{proposition}\label{minors}
  The ideal generated by the maximal minors of the matrix $M(g,f)$ is $\langle f
  \rem g \equiv 0 \rangle$.
\end{proposition}

\begin{proof}
  The ideal is unaffected by row and column operations.  The last column
  contains the coefficients of $f$ and the other columns contain the coefficients of the polynomials $x^i g$ for $n-m \ge i \ge 0$, so dividing $f$ by $g$ amounts to the unique column operations that
  replace the top $n-m+1$ entries of the last column by zeroes, leaving the coefficients of the
  remainder $f \rem g$ in the bottom part of that column.  Row operations can be then used to replace the
  entries below the diagonal in the other columns by zeroes, without affecting
  the entries in the last column.  The maximal minors that involve one row in
  the bottom part and all the rows in the top part are the coefficients of $f
  \rem g$.  The other minors vanish, because they involve two rows in the bottom
  part, whose entries outside the last column are zero.
\end{proof}

\def\matentry#1#2#3{\langle#2|#1|#3\rangle}

Equip the ring $R$ from Definition \ref{PFdef} with the {\em graded reverse lexicographic ordering} $<_{\rm{grevlex}}$  with the
variables listed in the order $c_{1} > \dots > c_{m}$ and with $\deg(c_{i}) =
i$.  That means that $$c^\alpha = c_1^{\alpha_1} \dots c_m^{\alpha_m} >
c^\gamma = c_1^{\gamma_1} \ldots c_m^{\gamma_m}$$ if and only if one of the following holds: $\sum_i i
\alpha_i > \sum_i i \gamma_i$, or $\sum_i i \alpha_i = \sum_i i \gamma_i$ and
$\alpha_\ell < \gamma_\ell, \alpha_{\l+1} = \gamma_{\l+1}, \dots,
\alpha_m = \gamma_m$, for some $\l$.

\begin{lemma}\label{leadtermlemma}
Let $f = x^n+ b_1 x^{n-1} + \dots + b_n, g(x)=x^m+c_1x^{m-1}+\dots+c_m$ and $I=\langle f\%g\equiv 0\rangle$.
The ideal $\l_{<_{\rm{grevlex}}}(I)$ contains each monomial $c^\beta$ with $\beta_1 + \dots +
  \beta_m = n-m+1$, i.e., $\l_{<_{\rm{grevlex}}}(I) \supseteq \langle c_{1}, \dots, c_{m}
  \rangle^{n-m+1}$.
\end{lemma}

\begin{proof}
  Applying Proposition \ref{minors}, we see that $I$ is generated by the maximal
  minors of the matrix $M = M[g,f]$ constructed above.  The ideal is unaffected
  by row and column operations, so we permute the columns cyclically so the
  last one becomes the first.  The resulting matrix has the following form if
  $n-m=4$.
  \[\begin{pmatrix}
             1  &  1 &   0 &   0&   0&   0  \\
             b_1&c_1 &   1 &   0&   0&   0  \\
             b_2&c_2 & c_1 &   1&   0&   0  \\
             b_3&c_3 & c_2 & c_1&   1&   0  \\
             b_4&c_4 & c_3 & c_2& c_1&   1  \\
             b_5&c_5 & c_4 & c_3& c_2& c_1  \\
             b_6&c_6 & c_5 & c_4& c_3& c_2  \\
             b_7&c_7 & c_6 & c_5& c_4& c_3  \\
   \vdots&\vdots&\vdots&\vdots&\vdots&\vdots\\
         b_{n-1}&  0 &   0 &   0& c_m& c_{m-1}  \\
             b_n&  0 &   0 &   0&   0& c_m  \\
        \end{pmatrix}.
  \]
  The number of columns is $n-m+2$, so from the form of the matrix one sees
  that any monomial $c^\beta$ of the desired form occurs as the product of
  diagonal elements in a unique maximal submatrix that includes the first row.
  For example, the monomial $c_4 c_5^2 c_7 c_9$ arises from the diagonal of the
  following submatrix, with its diagonal elements parenthesized for emphasis.
  \[\begin{pmatrix}
             (1)  &  1 &   0 &   0&   0&   0  \\
             b_4&(c_4) & c_3 & c_2& c_1&   1  \\
             b_6&c_6 & (c_5) & c_4& c_3& c_2  \\
             b_7&c_7 & c_6 & (c_5)& c_4& c_3  \\
          b_{10}&c_{10} & c_9 & c_8& (c_7)& c_6  \\
          b_{13}&c_{13} & c_{12} & c_{11}& c_{10}& (c_9)  \\
        \end{pmatrix}.
  \]
  We claim that $c^\beta$ is the lead term of the determinant of that
  submatrix.

  If it were true that $\deg b_i = i$, then all the terms of the determinant
  would have the same degree.  But $\deg b_i = 0$, because $b_i$ is a
  scalar, so the terms of the determinant that involve one of them have degree
  less than $\deg c^\beta$.  Hence we may ignore them when looking for the leading term with respect to $<_{\rm{grevlex}}$.  A convenient way to do
  so is to remove the first row and the first column.  Our example matrix would
  then look like this:
  \[\begin{pmatrix}
          (c_4) & c_3 & c_2& c_1&   1  \\
          c_6 & (c_5) & c_4& c_3& c_2  \\
          c_7 & c_6 & (c_5)& c_4& c_3  \\
          c_{10} & c_9 & c_8& (c_7)& c_6  \\
          c_{13} & c_{12} & c_{11}& c_{10}& (c_9)  \\
   \end{pmatrix}.
  \]
  Consider a nonzero term of the determinant arising from a non-identity
  permutation of the rows: it involves choosing one entry from each row and one
  entry from each column.  In our example matrix, it may involve the bracketed
  entries in the following matrix.
  \[\begin{pmatrix}
          c_4 & [c_3] & c_2& c_1&   1  \\
          [c_6] & c_5 & c_4& c_3& c_2  \\
          c_7 & c_6 & c_5& [c_4]& c_3  \\
          c_{10} & c_9 & [c_8]& c_7& c_6  \\
          c_{13} & c_{12} & c_{11}& c_{10}& [c_9]  \\
   \end{pmatrix}.
  \]

  All terms of the determinant have the same degree, so for comparing two of
  them in the graded reverse lexicographic ordering, one may simply say that
  $$c^\alpha = c_1^{\alpha_1} \dots c_m^{\alpha_m} > c^\gamma = c_1^{\gamma_1} \ldots c_m^{\gamma_m}$$
  if and only if for some $\l$
  $$\alpha_\ell < \gamma_\ell, \alpha_{\l+1} = \gamma_{\l+1}, \dots, \alpha_m = \gamma_m.$$ 

  Now focus on the bottom-most row whose chosen entry is off the diagonal.  It
  must be to the left of the diagonal, because the columns to the right are
  already occupied by lower choices.  In our example, it's the second row from
  the bottom, where we see that a factor $c_7$ on the diagonal has given way to
  a factor $c_8$ to the left of the diagonal; the reason that $c_4 c_5^2 c_7
  c_9 > c_3 c_6 c_4 c_8 c_9 $ is that in the passage from the first monomial to
  the second, the exponent of $c_8$ has increased while the exponents on $c_9,
  c_{10}, \dots$ have not decreased.

  The general observation that gives the result is the following.  Adopt the
  convention that $c_0 = 1$.  Consider two monomials $u = c_{a_0} c_{a_1} \dots
  c_{a_{n-m}}$ and $v = c_{b_0} c_{b_1} \dots c_{b_{n-m}}$.  Assume: (1) the
  indices of the monomial $u$ are non-decreasing: $a_0 \leq a_1 \leq \dots \leq
  a_{n-m}$; and (2) $a_s < b_s, a_{s+1} = b_{s+1}, \dots, a_m = b_m$ for some
  $s$.  Then, in the reverse lexicographic order, we have $u > v$.

  The observation holds because in the passage from $u$ to $v$ an instance of
  the variable $c_{a_s}$ is replaced by $c_{b_s}$, and thus the exponent of
  $c_{b_s}$ increases, while the exponents of $c_{b_s+1},c_{b_s+2},\dots$ do
  not decrease.  Indeed, the only exponents that might decrease are those of
  $c_{a_0}, c_{a_1}, \dots, c_{a_s}$, but $a_0 \le a_1 \le \dots \le a_s <
  b_s$, so they are not in a position to interfere.
\end{proof}

If $N$ is a graded $A$-module such that for each $i$, the $A$-module $N_i$ is
free and finitely generated, then the {\em Hilbert series of $N$ over $A$} is
the power series 
\[H_A(N,T) := \sum_i \rank_A N_i \cdot T^i.
\]

\begin{lemma}\label{countmonoms}
  For all $n \ge m \ge 0$,
  $$H_A(R / \langle c_{1}, \dots, c_{m} \rangle^{n-m+1},T) =
  \frac{(1-T^n)\cdot{\dots}\cdot(1-T^{n-m+1})}{(1-T^{m})\cdot{\dots}\cdot(1-T^{1})}$$
\end{lemma}

\begin{proof}
  Introduce notation $H_{n,m}(T)$ for the left side of the identity claimed above and
  $h_{n,m}(T)$ for the rational function on the  right side.  The Hilbert function $H_{n,m}(T)$ tallies
  the monomials $c^\alpha$ with $\sum_i \alpha_i \le n-m$ by degree.  A
  monomial is either divisible by $c_m$ or not, which gives the recursive formula $H_{n,m}(T) = T^m
  H_{n-1,m}(T) + H_{n-1,m-1}(T)$.  The analogous identity $h_{n,m}(T) = T^m
  h_{n-1,m}(T) + h_{n-1,m-1}(T)$ can be verified by a direct computation, so the lemma can be established by
  induction, based on the cases where $m=0$, where both sides equal $1$.
\end{proof}

\begin{remark}\label{rankremark}
  The lemma implies that the right side of the equation is a polynomial, which
  is analogous to the statement that the binomial coefficient $\binom n m$ is
  an integer, the proof of which involves observing that a subset of
  $\{1,\dots,m\}$ either contains $m$ or not.  Factoring out all the factors of
  $1-T$ and then replacing $T$ by $1$ yields the equation $H_{n,m}(1) = \binom n m$,
  which is a standard fact about the number of monomials $c^\alpha$ with
  $\sum_i \alpha_i \le n-m$.
\end{remark}

\begin{proposition}\label{proposition:meatyprop}
  Let $f(x) = x^{n} + b_1 x^{n-1} + \dots + b_{n-1} x + b_{n} \in A[x]$ be a
  monic polynomial of degree $n$. Take $m \in \NN$ with $m \le n$, let $g(x)=x^{m} + c_1 x^{m-1} + \dots + c_{m-1} x + c_{m}$ and form as in definition \ref{PFdef}
  \[\PF_m(f)=A[c_i \vert 1\leq i \leq m]/\langle f\%g \equiv 0\rangle.\]
    Equip $R=A[c_i \vert 1\leq i \leq m]$ with the graded reverse lexicographic ordering  with
  the variables listed in the order $c_{1} > \dots > c_{m}$ and with
  $\deg(c_{i}) = i$. Set $I=\langle f\%g \equiv 0\rangle$ Then:
  \begin{enumerate}
    \item \label{meatyprop:leadterms} $\l(I) = \langle c_{1}, \dots, c_{m} \rangle^{n-m+1}$;
    \item \label{meatyprop:monic} $I$ is a monic ideal;
    \item \label{meatyprop:rank} $\PF_m(f)$, as an $A-$module, is free, with rank $\binom n m$; and
    \item \label{meatyprop:inject} the map $A \to \PF_m(f)$ is injective.
  \end{enumerate}
\end{proposition}

\begin{proof}
  Using Remark \ref{rankremark} we note that claim \ref{meatyprop:leadterms} implies  the rest of the claims, so we focus on it.  Lemma \ref{leadtermlemma} shows
  that $\l(I) \supseteq \langle c_{1}, \dots, c_{m} \rangle^{n-m+1}$, which
  is half of it.  Choose $G \subset I$ with $\langle \l(G) \rangle = \langle
  c_{1}, \dots, c_{m} \rangle^{n-m+1}$. We shall prove that $G$ is a \Gb{} for $I$.

  By Lemma \ref{lemma:specialization} we may pass to the universal situation by
  replacing $R$ by the ring $\ZZ[b_1,\dots,b_{n},c_1,\dots,c_m]$, promoting the
  ring elements $b_{i}$ to variables.  To ensure that the lemma applies, as
  monomial ordering we take the block ordering $[b]\gg[c]$ that restricts to
  the graded reverse lexicographic ordering on $[b]$ and on $[c]$ (see
  Definition \ref{prodorddef}).  Assigning degrees $\deg(b_{j}) = j$ and $\deg(c_j)= j$ makes $I$ homogeneous, as shown in Lemma \ref{lemma:homogeneous}. This implies that the elements of $G$ can be chosen to be homogeneous as well.

  By Lemma \ref{lemma:gbfields} we may replace working in  the ring $\ZZ[b_1,\dots,b_{n},c_1,\dots,c_m]$ by working in  $\QQ[b_1,\dots,b_{n},c_1,\dots,c_m]$.  In view of this, we reset our notation to $R=\QQ[b_1,\dots,b_{n},c_1,\dots,c_m]$ and $I=\langle f\%g\equiv 0\rangle \subset R$. In this
  situation, a finitely generated graded $R$-module $N$ has a Hilbert series
  $H_\QQ(N,T) := \sum_i (\rank_\QQ N_i) T^i$ in the variable $T$.  We will show that
  $\l(I) = \langle \l(G) \rangle$ by comparing the Hilbert series of the two ideals.

 The low-order coefficients $b_{n-m+1}, b_{n-m+2}, \dots,b_{n-1}, b_{n}$ of $f(x)$ don't participate in the division $f(x)\rem g(x)$, so each one of the coefficients $b_i$ with $n-m+1\leq i\leq n$ appears as an isolated term in the equation of degree $i$ arising from  $f(x)\rem g(x)\equiv 0$ and in no other equations.  Eliminating those variables eliminates all the equations,
  too, so we see that $R/I \cong \QQ[b_{1},\dots,b_{n-m},c_1,\dots,c_m]$,
  allowing its Hilbert series to be determined: $$H_\QQ(R/I,T) = \frac{1}
  {(1-T^{n-m}) \dots (1-T^1)\cdot(1-T^m) \dots (1-T^1)}.$$ From the observation
  in Lemma \ref{prop:equivGB} that $\red_I R$ is the free graded
  $A$-submodule of $R$ generated by the monomials not contained in $\l(I)$, we
  see that $H_\QQ(R/I,T) = H_\QQ(R/\l(I),T)$.
We use Lemma \ref{lemma:gbfields} for the first and penultimate equalities and Lemma \ref{countmonoms} for the intermediate steps to see that
  \begin{align*}
    H_\QQ(R/\langle \l(G) \rangle,T) & = H_\QQ(\QQ[b_1,\dots,b_{n}],T) \cdot H_{\QQ[b_1,\dots,b_{n}]}(R/\langle \l(G) \rangle,T) \\
    & = \frac{1}{(T^n-1) \cdot \dots \cdot (T-1)} \cdot
      \frac{(1-T^n)\cdot{\dots}\cdot(1-T^{n-m+1})}{(1-T^{m})\cdot{\dots}\cdot(1-T^{1})} \\
    & = \frac{1} {(1-T^{n-m}) \dots (1-T^1)\cdot(1-T^m) \dots (1-T^1)} \\
    & = H_\QQ(R/I,T) = H_\QQ(R/\l(I),T)
  \end{align*}

  It follows, since the ideals
  $\langle\l(G)\rangle \subseteq \l(I)$ have the same Hilbert series over the
  field $\QQ$, that  they are equal, establishing claim
  \ref{meatyprop:leadterms}.
\end{proof}

\begin{corollary}\label{meatyprop:Gb} 
 With respect to the monomial ordering in Proposition \ref{proposition:meatyprop}, the set of maximal minors of the matrix $M(g,f)$  is a minimal \Gb{} of the ideal of relations of  $\PF_m(f)$.
\end{corollary}
\begin{proof}
Recalling Lemma \ref{leadtermlemma}, we see that  the proof of part \ref{meatyprop:leadterms} of Proposition \ref{proposition:meatyprop} given above implies this corollary.
\end{proof}

The \Gb{} of the corollary is often not auto-reduced.

\begin{remark}
  The ideal $I$, regarded as an ideal in the ring $\ZZ[b,c]$, provides an example of an ideal
  the size of whose \Gb{} depends on the choice of monomial ordering.  With the
  block ordering $[c]\gg[b]$ used in the proof of Proposition \ref{proposition:meatyprop}, the size of a minimal \Gb{} is the number of monomials in $m$-variables of degree $n-m+1$, or $\binom n
  {m-1}$.  With the blocks in the other order $[b]\gg[c]$, the leading terms
  are $b_{n-m+1}, b_{n-m+2}, \dots, b_{n-1}, b_{n}$, so the size of the \Gb{}
  is $m$. Both block orderings mentioned here are induced by the graded reverse lexicographic orderings on the blocks $[b]$ and $[c]$.
\end{remark}

Now we consider the general universal factorization.

\begin{theorem}\label{theorem:ourtheorem}
  Given a monic polynomial $f(x) \in A[x]$ of degree $n$ and integers $r, n_1, \dots, n_r \in \NN$ with $n_1 + \dots + n_r \le n$, let $\PF_{n_1,\dots,n_r}(f) :=R/I$, where $R=A[c_{ij} \vert 1 \leq i \leq r,1 \leq j \leq n_i], g_i = x^{n_i} + c_{i1}x^{n_i-1} + \cdots + c_{in_i} \in R[x]$ for $1 \leq i \leq r$, and  $I = \langle f \rem g_1 g_2 \dots
  g_r \equiv 0 \rangle \subseteq R$.
  Set $\deg(c_{ij}) = j$, equip each monoid
  $[c_{i1}, \dots, c_{in_i}]$ with the graded reverse lexicographic ordering $<_i$
  where $c_{i1} > \dots > c_{in_i}$, and finally equip the monoid $[c]$ of the
  polynomial ring $R$ with the block ordering denoted $<_{r,\ldots,1}$ where we compare according to $[c_{r1}, \dots, c_{rn_r}]\gg[c_{r-1,1}, \dots, c_{r-1,n_{r-1}}]\gg \ldots \gg[c_{11}, \dots, c_{1n_1}]$ (see Definition \ref{prodorddef}). 
  Then:
  \begin{enumerate}
    \item \label{ourtheorem:monic} $I$ is monic over $A$; 
    \item \label{ourtheorem:leadterms} the lead term ideal $\l(I)$ is
      \begin{align*}
        \l_{<_{r,\ldots,1}}( I)   = {}& \phantom{{}+{}} \langle  c_{11}, \dots, c_{1n_1} \rangle{}^{n -  n_1 + 1}   \\
         & + \langle c_{21}, \dots, c_{2n_2}\rangle{}^{n - n_1 - n_2 + 1} \\
         & + \dots \\
         & + \langle c_{r1}, \dots, c_{rn_r}\rangle{}^{n - n_1-n_2-\dots- n_r + 1}.
      \end{align*}
    \item \label{ourtheorem:basis} an explicit basis for $\PF_{n_1,\dots,n_r}(f)$ as an $A$-module is provided by the images of the monomials
      $c^\alpha$ satisfying the following inequalities:
      \begin{align*}
       \alpha_{11} + \dots + \alpha_{1n_1} & \le n - n_r- \dots -n_1 \\
       \alpha_{21} + \dots + \alpha_{2n_2} & \le n - n_r-  \ldots -n_{2} \\
       & \hskip 5pt \vdots \\
       \alpha_{r1} + \dots + \alpha_{rn_r} & \le n  - n_r;
      \end{align*}
 \item \label{ourtheorem:rank} $R/I$ is a free $A-$module of rank 
      $$\binom {n} {n_r} \cdot \binom {n - n_r} {n_{r-1}} \cdot 
      \binom {n - n_r - n_{r-1}} {n_{r-2}} \cdot \dots \cdot \binom {n - n_r - \dots - n_{2}} {n_1}.$$
  \end{enumerate}
\end{theorem}

\begin{proof}
  The idea of the proof is to construct $\PF_{n_1,\dots,n_r}(f)$ as a tower with $r$ stages, by letting
  $f$ acquire one factor at a time. We now describe the details. First we adjoin variables
  $c_{r1},\dots,c_{r,n_r}$ to $A$, set $g_r(x)=x^{n_r} + c_{r1} x^{n_r-1} + \dots + c_{r,n_r}$ and form the quotient $B_1=A[c_{r1},\dots,c_{r,n_r}]/\langle f \rem g_r(x) \equiv 0 \rangle$.  Next we
  adjoin the variables $c_{r-1,1},\dots,c_{r-1,n_{r-1}}$ to $B_{1}$, we set $g_{2}(x)=x^{n_{r-1}} + c_{r-1,1} x^{n_{r-1}-1} + \dots +c_{r-1,n_{r-1}}$ and we form the new
  quotient  $$B_{2}=B_1[c_{r-1,1},\dots,c_{r-1,n_{r-1}}]/\langle ( f \quo g_r(x)) \rem g_{r-1}(x) \equiv 0 \rangle$$  Inductively, denoting for $1\leq i\leq r$ $f_i(x)=f \quo g_r(x)g_{r-1}(x)\ldots g_{i-1}(x)\in B_{i-1}[x]$,  so that $\deg(f_i)=n-n_r-n_{r-1}-\ldots-n_{i-1}$, we construct the ring 
  $$B_i=B_{i-1}[c_{r-i+1,1},\dots,c_{r-i+1,n_{i}}]/\langle  f_i \rem g_{i} \equiv 0 \rangle.$$

  A monic \Gb{} for the ideal  $$\langle  f_i \rem g_{i} \equiv 0 \rangle \subset B_{i+1}[c_{r-i+1,1},\dots,c_{r-i+1,n_{i}}]$$ can be obtained as in the proof of Proposition \ref{proposition:meatyprop}. Thus using the graded reverse lexicographic order on the monoid $[c_{r-i+1,1},\dots,c_{r-i+1,n_{i}}]$, denoted $<_i$, yields 
  $$\l_{<_i}\left(\langle  f_i \rem g_{i} \equiv 0 \rangle\right)=\langle c_{r-i+1,1},\dots,c_{r-i+1,n_{i}}\rangle{}^{n - n_r - n_2 \dots - n_{i} + 1} .$$
  
   These \Gbs{} can be composed using Proposition \ref{composite}, finishing the proof of the statements \ref{ourtheorem:monic} and \ref{ourtheorem:leadterms}. Then statement \ref{ourtheorem:basis} follows from \ref{ourtheorem:leadterms} by means of Lemma \ref{prop:equivGB} and statement \ref{ourtheorem:rank} follows from \ref{ourtheorem:leadterms} by a standard combinatorial argument. 
\end{proof}

A direct corollary of the above proof gives a concrete method for finding a \Gb{} for our ideals of interest.

\begin{corollary}
An explicit \Gb{} for the defining ideal of $\PF_{n_1,\dots,n_r}(f)$ with respect to the block ordering in Theorem \ref{theorem:ourtheorem} can be obtained by composition, as in Proposition \ref{composite}, from the \Gbs{} of the ideals $\langle  f_i \rem g_{i} \equiv 0 \rangle$ defined in the proof of Theorem \ref{theorem:ourtheorem}. These \Gbs{} are given by the maximal minors of the matrices $M(f_i,g_i)$ (see Corollary \ref{meatyprop:Gb}).
\end{corollary}

\section{Monic homomorphisms}\label{section:further-formalities}

In this section we prepare for another proof of Proposition \ref{proposition:meatyprop}, to
be presented in Section \ref{section:another-proof}. We include this alternate proof in part
    because the techniques used may have other applications.

%

\begin{definition}\label{definition:monicmap}
  Let $S = A[y] = A[y_1,\dots,y_n]$ and $R = A[x] = A[x_1, \dots, x_m]$ be
  polynomial rings equipped with monomial orderings $<_x$ and $<_y$ respectively, and let $\ph : S \to R$ be
  a homomorphism of $A$-algebras.  We say that $\ph$ is {\em monic} if the
  following conditions are satisfied.
  \begin{enumerate}
    \item \label{item:monicmap:monic} for every $i$, the polynomial $\ph(y_i)$ is
      monic;
    \item \label{item:monicmap:injec} the homomorphism $\tph : [y] \to [x]$ of monoids defined by $
      \tph (y^\beta) = \l_{<_x} (\ph (y^\beta)) $ is injective; and
    \item \label{item:monicmap:ordpres} $\tph$ is order preserving.
  \end{enumerate}
\end{definition}

Observe that part \ref{item:monicmap:monic} can be used to prove that the function
$\tph$, defined in \ref{item:monicmap:injec}, is a homomorphism.  This follows since  the lead term of the product of two monic polynomials is the product of the two lead terms.

Injectivity of $\tph$ ensures that $\ph$ is injective.  The injectivity also
implies that $\l_{<_x}(\ph(g)) = \l_{<_x}(\ph(\l_{<_y}(g)))$ for any polynomial $g \in S$. 

\begin{remark}
  In Definition \ref{definition:monicmap}, suppose $S$ is not equipped with a monomial
  ordering $<_y$, so that part \ref{item:monicmap:ordpres} has no meaning, and suppose that the homomorphism  $\ph : S \to R$ satisfies parts \ref{item:monicmap:monic} and \ref{item:monicmap:injec} of Definition \ref{definition:monicmap}. Then the injectivity of $\tph$ ensures that the ordering on $[y]$ induced
  by the ordering on $[x]$ via $\tph$ is a monomial ordering.  Here by the ordering on $[y]$ induced by $\tph$ we mean the order $<_{\tph}$ defined by $y^\alpha<_{\tph} y^\beta$ if and only if $\tph(y^\alpha)<_x \tph(y^\beta)$. 
Endowing $[y]$ with that monomial ordering makes part \ref{item:monicmap:ordpres} of Definition \ref{definition:monicmap} true.
\end{remark}

\begin{remark}\label{remark:monic}
  The monic homomorphism that we will use in Section \ref{section:another-proof} is the
  following.  Let $S = A[y_1, \dots, y_r]$ and $R = A[x_1, \dots,
    x_r]$. Equip $R$ with a monomial ordering $<_x$ such that $x_1 > x_2 > \dots >
  x_r$.  Define $\ph : S \to R$ by letting $\ph(y_i)$ be the $i$-th elementary
  symmetric polynomial of the variables $x_1, \dots, x_r$.  Then $\ph$ is
  monic, because $\tph(y_i) = \l_x(\ph(y_i)) = x_1 x_2 \dots x_i$ is monic and
  the map $\tph : [y] \to [x]$ is injective.  In particular, if the monomial ordering on $R$
  is the graded reverse lexicographic order with all variables of degree $1$,
  and $x_1 > \dots > x_r$, then the proof of Proposition \ref{proposition:meatyprop} describes the induced monomial ordering on S; it is the graded reverse lexicographic
  order with $\deg y_i = i$ and $y_1 > y_2 > \dots > y_r$.
\end{remark}

\begin{lemma}\label{lemma:one}
 Let $S = A[y] = A[y_1,\dots,y_n]$ and $R = A[x] = A[x_1, \dots, x_m]$ be
  polynomial rings equipped with monomial orderings $<_x$ and $<_y$ respectively and let $\ph:S \to R$ be a monic homomorphism of $A$-algebras.  Suppose $I$ is a monic ideal
  of $R$ with $\red_I \ph(S) \subseteq \ph(S)$, and let $J=\ph^{-1}(I)$.  Then
  \begin{enumerate}
  \item \label{one:monoms} a monomial $y^\beta \in S$ is in $\l(J)$ if and only
    if $\tph(y^\beta)$ is in $\l(I)$;
  \item \label{one:monic} $J$ is monic.
  \end{enumerate}
\end{lemma}

\begin{proof}
  It will suffice to prove that a term $a y^\beta \in S$ with $a \ne 0$ is in
  $\l(J)$ if and only if $a \tph(y^\beta)$ is in $\l(I)$: part
  \ref{one:monoms} follows by setting $a=1$; and part \ref{one:monic}
  follows because $I$ is monic, and hence the condition $a \tph(y^\beta) \in
  \l(I)$ is independent of $a$.

  We prove the two implications in our claim separately.

  Assume $a y^\beta \in \l(J)$.  Locate a polynomial $a y^\beta - g \in J$ with
  lead term $a y^\beta $.  Applying $\ph$ yields $a \ph(y^\beta) - \ph(g) \in
  I$, and its lead term is $a \tph(y^\beta)$, showing $a \tph(y^\beta) \in
  \l(I)$.

  Alternatively, assume $a \tph(y^\beta) \in \l(I)$.  Then the lead term of $\ph(a y^\beta)$ is a lead term of $I$, and hence $\ph(a y^\beta) > \red_I \ph(a y^\beta) $.  By hypothesis, we can write $\red_I \ph(a y^\beta) = \ph(g)$ for some  (unique) $g \in S$.  Then $\ph(a y^\beta - g) = \ph(a y^\beta) - \red_I \ph(a y^\beta) \in I$, so $a y^\beta - g \in J$,  with $y^\beta$ bigger  than any monomial in $g$, showing $a y^\beta \in \l(J)$.
\end{proof}

We will apply the following result in Section \ref{section:another-proof}.  Although the
conditions on $L$ in the following proposition seem numerous, these conditions will all be straightforward
to verify.

\begin{proposition}\label{proposition:cond}
 Let $S = A[y] = A[y_1,\dots,y_n]$ and $R = A[x] = A[x_1, \dots, x_m]$ be
  polynomial rings equipped with monomial orderings $<_x$ and $<_y$ respectively.  Suppose that $\ph:S \to R$ is a monic $A$-algebra homomorphism, $I$ is a monic ideal of $R$, and set $J = \ph^{-1}(I)$. Assume that $S/J$ is a finitely  generated projective $A$-module of finite rank $N$ and  that $L$ is a finitely generated $A$-submodule of $R$ satisfying:
  \begin{enumerate}
    \item \label{cond:phsb} $L \subseteq \ph(S)$
    \item \label{cond:redI} $L \subseteq \red_I(R)$ 
    \item \label{cond:sumd} $L$ is a summand of $R$, regarded as an $A$-module,
    \item \label{cond:rank} $\rank L =N$.
  \end{enumerate}
  Then $L = \red_I \ph(S)$, and
  \begin{enumerate}
    \item[a.] $J$ is monic,
    \item[b.] $S/J$ is a free $A$-module of rank $N$, and
    \item[c.] $\l(J) = \{ y^\beta \mid \tph (y^\beta) \in \l(I) \}$.
  \end{enumerate}
\end{proposition}

\begin{remark}
  Notice that \ref{cond:sumd} implies $L$ is a finitely generated projective $A$-module, so $\rank L$ is well defined, allowing its use in \ref{cond:rank}.
\end{remark}

\begin{remark}
  A slightly more general statement can be made when $\rank L$ is infinite by considering graded rings and assuming that each graded
  component of $S/J$ is of finite known rank.
\end{remark}

\begin{proof}
  The $A$-linear isomorphism $R/I \isomap \red_I R$ induces an $A$-linear isomorphism $S/J \isomap
  \red_I \ph(S)$.  By part \ref{cond:phsb} of the hypothesis, we know that $\red_I L \subseteq \red_I \ph(S)$ and by part \ref{cond:redI} we know that $L = \red_I L$, since every element of $L$ is reduced modulo $I$. Thus we obtain $L  \subseteq \red_I \ph(S)$.  From condition \ref{cond:sumd} of the hypothesis we infer that $L$ is also a summand of $\red_I \ph(S)$.  From  condition \ref{cond:rank} of the hypothesis, we can
deduce that $L$ and $\red_I \ph(S)$ are finitely generated projective modules of the same rank, hence equal.  Finally, using condition \ref{cond:phsb} we  conclude that $\red_I \ph(S) \subseteq \ph(S)$, which is the hypothesis needed to apply  Lemma \ref{lemma:one}. This lemma yields that $J$ is monic, and therefore that $S/J$ is free over $A$.  The last statement in the conclusion follows from part \ref{one:monoms} of Lemma \ref{lemma:one}.
\end{proof}

\subsection{Alternate construction of \Gbs{} defined via monic morphisms}

Let $\ph : S \to R$ be a monic morphism.  Let $I \subseteq R$ be a monic ideal,
and let $J = \ph^{-1}(I)$. Suppose we know how to compute $\red_I$ (e.g., we
might know a \Gb{} of $I$).  In the rest of this section, we discuss a way to
compute $\red_J$ (with respect to the induced monomial ordering on $S$) without
first constructing a \Gb{} of $J$, at least in some cases when $\ph$ and $I$
satisfy extra conditions.

\begin{proposition}\label{star}
 Let $S = A[y] = A[y_1,\dots,y_n]$ and $R = A[x] = A[x_1, \dots, x_m]$ be
  polynomial rings equipped with monomial orderings $<_x$ and $<_y$ respectively.  Suppose that $\ph:S \to R$ is a monic $A$-algebra homomorphism, $I$ is a monic ideal of $R$, $J = \ph^{-1}(I)$, and $\red_I \ph(S) \subseteq \ph(S)$.
  Suppose also that the following condition is satisfied.
  \begin{enumerate}
    \item[($\ast$)] for any monomial $y^\beta$ of $S$, if any term of
      $\ph(y^\beta)$ is in $\l(I)$, then the lead term $\l(\ph(y^\beta))$ is in $\l(I)$.
  \end{enumerate}
  Then $\red_I \circ \ph = \ph \circ \red_J$.
\end{proposition}

\begin{remark}
  Under the hypotheses, Lemma \ref{lemma:one} applies and states that $J$ is
  monic, hence $\red_J$ is defined, allowing its use in the claim.
\end{remark}

\begin{proof}
  Consider a polynomial $g \in S$.  We wish to prove $\red_I(\ph(g)) = \ph(\red_J g)$.  Replacing $g$ by
  $\red_J g$ does not alter either side of the equation, so we may assume $g$ is reduced modulo $J$, and try to show that
  $\ph(g)$ is reduced modulo $I$.  The polynomial $g$ is reduced if and only if each of its terms is, so we may assume
  that $g$ is a monomial: $g = y^\beta \notin \l(J)$.  If $\ph(y^\beta)$ were not reduced modulo $I$, then by
  hypothesis, we would know that $\tph(y^\beta) \in \l(I)$, contradicting statement \ref{one:monoms} of Lemma \ref{lemma:one}.
\end{proof}

\begin{remark}\label{fastelim}
  The computation of $\ph^{-1}$ may be fast enough in practice for it to be
  feasible, in light of Proposition \ref{star}, to compute $\red_J$ as
  $\ph^{-1} \circ \red_I \circ \ph$.

  We often know $\l(J)$, e.g., from Lemma \ref{lemma:one}.  If $y^\beta \in \l(J)$,
  the corresponding \Gb{} element is $y^\beta - \ph^{-1} \circ \red_I \circ
  \ph(y^\beta)$.  If $\ph^{-1}$ can be computed fast enough, this can be a
  quick way to determine a minimal 
  \Gb{} of $J$.

  One way to compute $\ph^{-1}$ is via a \Gb{} with an elimination order. In
  the polynomial ring $A[x,y]$, consider the ideal $L = \langle y_i - \ph(y_i)
  \rangle$.  Choose a block ordering $[x] \gg [y]$ (see Definition
  \ref{prodorddef}).  If $L$ is monic (or, if $A$ is a field), then
  $\ph^{-1}(f) = \red_L(f) \in A[y]$ if $f$ is in the image of $\ph$.
\end{remark}

\section{Another proof}\label{section:another-proof}

In this section we give another proof of Proposition \ref{proposition:meatyprop}.
The following lemma discusses $\PF_1(f)$, the simplest case.

\begin{lemma}\label{lemma:onelinear}
  Let $f(x) \in A[x]$ be a monic polynomial of degree $n$ in one variable, and let $R= A[c_{ij} \vert 1 \leq i \leq r,1 \leq j \leq n_i], I=\langle f \%(x+c_{11})\rangle$ and set $ \PF_1(f):=R/I$ as in definition \ref{PFdef}.  Then:
  \begin{enumerate}
    \item \label{onelinear:a} $I$ is generated by $f(-c_{11})$ and is monic;
    \item \label{onelinear:b} $\l(I) = \langle c_{11}^n \rangle$ relative to the unique monomial ordering on $R$;
    \item \label{onelinear:c} the algebra $\PF_1(f)$ is monic and is free with rank $n$ as an $A$-module;
    \item \label{onelinear:d} the quotient $f/(x+c_{11})$ in $\PF_1(f)[x]$ is the image
      of $f \quo (x+c_{11}) \in R[x]$; and
    \item \label{onelinear:e} if $n \ge 1$, then the map $A \to \PF_1(f)$ is injective.
  \end{enumerate}
\end{lemma}

\begin{proof}
  One observes that $f \rem (x+c_{11}) = f(-c_{11})$, which, up to sign, is
  a monic polynomial of degree $n$ in the variable $c_{11}$, establishing claim
  \ref{onelinear:a}.  The polynomial ring $R$ has just one variable,
  $c_{11}$, so its unique monomial ordering is the one with $1 < c_{11} <
  c_{11}^2 < \cdots$, establishing claim \ref{onelinear:b}.  The other claims
  follow.
\end{proof}

The next lemma considers $\PF_{m:1}(f)$.  The resulting basis over $A$ is
well known; see, for example, \cite[Proposition 1.10]{laksov-MR2130463}.

\begin{lemma}\label{lemma:case1}
  Let $f(x) \in A[x]$ be a monic polynomial of degree $n$ in one variable, take $m \in \NN$
  with $m \le n$, and let $ \PF_{m:1}(f):=R/I$ as in definition \ref{PFdef}.  Equip $R$ with any monomial
  ordering satisfying $c_{11} > c_{21} > \dots > c_{m1}$.  Then:
  \begin{enumerate}
    \item $I$ is a monic ideal;
    \item $\l(I) = \langle c_{11}^{n-m+1}, c_{21}^{n-m+2}, \dots, c_{m1}^n \rangle$;
    \item as an $A$-module, $\PF_{m:1}(f)$ is free, with rank $n (n-1) \dots (n-m+1)$; and
    \item the map $A \to \PF_{m:1}(f)$ is injective.
  \end{enumerate}
  Furthermore, a monic \Gb{} for $I$ can be described explicitly as shown in the proof (the \Gb{} consists of the polynomials $\{Q_1,\ldots, Q_m\}$ defined therein). 
\end{lemma}

\begin{proof}
  We construct $\PF_{m:1}(f)$ as a tower with $m$ stages, by letting $f$
  acquire one linear factor at a time, so that each stage is described by Lemma
  \ref{lemma:onelinear}, and so their \Gbs{} can be composed with Proposition
  \ref{composite}.  First we adjoin the variable $c_{m1}$ and form the quotient $B_1=A[c_{m,1}]/I_1$
  where $I_1=\langle f \rem (x+c_{m,1}) \equiv 0 \rangle$.  Then we adjoin
  the variable $c_{{m-1},1}$ and form the quotient $B_2=B_1/I_2$, where  
  $$I_{2}=\langle (f \quo (x+c_{m,1})) \rem (x+c_{{m-1},1}) \equiv 0 \rangle\subset B_1.$$  
 Inductively we continue adjoining at each step $1\leq i\leq m$ the variable $c_{{m-i},1}$ to $B_{i}$, we set
 $f_{i+1}(x)=  (f \quo((x+c_{m,1})\dots(x+c_{{m-i+1},1}))) \rem (x+c_{{m-i},1}) $ and form the quotient $B_{i+1}=B_i/I_{i+1}$ by the ideal $I_{i+1}=\langle f_{i+1} \equiv 0 \rangle$. With an eye towards applying Proposition
  \ref{composite}, we first replace the polynomials $f_i$ by monic polynomials. 
 For $i = 1, \dots, m$, write $$Q_i= (-1)^{n-i+1} (f \quo
  ((x+c_{m,1})\dots(x+c_{{m-i+2},1}))) \rem (x+c_{{m-i+1},1}) \in B_i$$ and notice that $Q_i$  involves
  just the variables $c_{{m-i+1},1}, \ldots, c_{m1}$.  The polynomials $Q_i$ are
 the same as the polynomials $f_i$ defining the rings $B_1, \ldots, B_m$ above, with the sign
  altered to ensure that the coefficient of the lead monomial
  in $Q_i$ is $1$. It is shown below that this lead monomial is  $c_{{m-i+1},1}^{n-i+1}$.
  
  Proposition \ref{composite} provides all four parts of our claim, if the
  monomial ordering is the block ordering $[c_{11}] \gg \dots \gg [c_{m1}]$,
  which is also called the lexicographic ordering.  In this ordering, the lead
  monomial of $Q_i$ is $c_{{m-i+1},1}^{n-i+1}$. Indeed, this follows by part 1 of Lemma \ref{lemma:case1}, noticing that $Q_i$ is obtained by evaluating the polynomial $(f \quo((x+c_{m1})\dots(x+c_{{m-i+2},1})))$ at $x=-c_{{m-i},1}$. Since the quotient polynomial has degree $n-i+1$ and is monic, it follows that $c_{{m-i+1},1}^{n-i+1}$ is a term in $Q_i$. Part 2 of Lemma \ref{lemma:case1} shows that $\l_{<_{[c_{m-i+1}]}}(Q_i)=c_{{m-i+1},1}^{n-i+1}$, but since the monomials appearing in $Q_i$ are in the monoid generated by $[c_{m-i+1},\ldots,c_{m1}]$ and the block ordering gives priority to $[c_{{m-i+1},1}]$, it follows that $\l_{<_{\rm{lex}}}(Q_i)=\l_{<_{[c_{m-i+1}]}}(Q_i)=c_{{m-i+1},1}^{n-i+1}$. Composing the \Gbs{} of the ideals $I_i$ using Proposition \ref{composite} yields that  the ideal $I=\langle Q_1, \dots, Q_m \rangle$ and that the latter set of generators forms a \Gb{} for $I$.  (In the case where $m = n$, this \Gb{} is also presented in the first chapter of \cite{MR1255980}.)

  To see what happens if we use a different monomial ordering, we determine the
  monomials appearing in $Q_i$.  For that purpose, we temporarily consider the
  universal example of a monic polynomial $f$ of degree $n$, which has the form
  $f = x^n + b_1 x^{n-1} + \dots + b_n \in \ZZ[b_1,\dots,b_n][x]$.  Recall from Lemma \ref{lemma:homogeneous} 
  that if we assign degrees to all the variables, so that $x$ and each $c_{i1}$ have degree $1$,
  and $b_i$ has degree $i$, then $f$ and all the
  polynomials appearing above are homogeneous.  In particular, $Q_i$ is
  homogeneous of degree $n-i+1$.  Letting the degree of $x$ and $b_1, \dots, b_n$ revert to
  their former values, which are of degree $0$, we see that each term of $Q_i$
  has degree at most $n-i+1$ in the variables $c_{m-i+1,1},\ldots,c_{m,1}$ and
  contains the term $c_{{m-i+1},1}^{n-i+1}$ of degree equal to $n-i+1$.
 
 Now we see that if we use another monomial ordering satisfying $c_{11} >
  \dots > c_{m1}$, then the lead terms of the $Q_i$ remain the same, since $c_{{m-i+1},1}^{n-i+1}$ is larger than any monomial of degree at most $n-i+1$ in $[c_{m-i+1,1},\ldots,c_{m,1}]$. 
  Therefore with respect to this new ordering, $\{Q_1, \dots, Q_m\}$ remains a
  \Gb{} and $I$ remains monic over $A$, by Lemma \ref{lemma:ordering}.
\end{proof}

Here is our second proof of Proposition \ref{proposition:meatyprop}, in which we 
compute the leading term ideal for $\PF_m(f)$.

\begin{proof}[Second proof of Proposition \ref{proposition:meatyprop}]

Recall that in the statement of Proposition \ref{proposition:meatyprop} we have 
\begin{eqnarray*}
f(x) &=& x^n+b_1x^{n-1}+\ldots+b_n\in A[x]  \text{ and } \\
g(x) &=& x^m + c_{1}x^{m-1} + \dots + c_{m}\in A[c_1,\ldots,c_m][x].
\end{eqnarray*}
 In order to analyze the splitting algebra $\PF_m(f)$ in which the polynomial $f$ acquires a factor $g$ of degree $m$, we shall construct a map $\bar\ph: \PF_m(f) \to \PF_{m:1}(f)$ and show that it factors though an algebra $\PF_{m:1}(\bar{g})$ defined below. Towards this end let
\begin{equation*}
S =A[c_i \vert 1\leq i\leq m], \
R=A[c_{i1} \vert 1\leq i\leq m], \text{ and }
R' =A[c_i,c_{i1} \vert 1\leq i\leq m]
\end{equation*}
and set  $g_i = x^{n_i} + c_{i1}x^{n_i-1} + \cdots + c_{in_i} \in
  R[x]$ as in definition \ref{PFdef}. 
For each of the rings above consider the ideals
\begin{eqnarray*}
J &=& \langle f \rem g \equiv 0 \rangle\subset S, \\
I &=& \langle f\rem g_1\ldots g_m \equiv 0 \rangle\subset R, \text{ and }\\
I' &=& \langle f\rem g \equiv 0, g\rem g_1\ldots g_m \equiv 0 \rangle\subset R'.
 \end{eqnarray*}
  With this notation, by Definition \ref{PFdef} we have  $\PF_{m}(f)=S/J$ and $\PF_{m:1}(f)=R/I$. 
  Consider the map $\ph : S \to R$ that sends $g(x)=x^m + c_{1}
  x^{m-1} + \dots + c_{m} \in S[x]$ to $\prod_{i=1}^m (x+c_{i1}) \in R[x]$ by sending each variable $c_i\in S$ to the $i$-th symmetric polynomial in the variables $c_{11},\ldots,c_{m1}$ of $R$. It is easy to see that this map satisfies $\ph(J)\subseteq I$, and thus that $\ph$ induces a ring homomorphism $\bar\ph:S/J\to R/I$. 
  
    The algebra $\PF_m(f)$ comes equipped with a monic factor $\bar g(x) \in \PF_m(f)[x]$ of $f(x)$ with  $\deg(\bar g)=m$, where $\bar g(x)$ is the image of $g(x)$ in $R'/I'$.  The definition of $\PF_{m:1}(\bar g)$ together with $\PF_m(f)=R/I$ give the following explicit description
\begin{eqnarray*}
\PF_{m:1}(\bar g) &=& \PF_m(f)[c_{i1}\vert 1\leq i\leq m]/\langle \bar g\rem g_1\ldots g_m \equiv 0 \rangle\\
&\cong& R[c_{i1} \vert 1\leq i\leq m]/\langle f\rem g \equiv 0, g\rem g_1\ldots g_m \equiv 0 \rangle= R'/I'.
\end{eqnarray*}

Next we show that the homomorphism $\bar\ph:S/J\to R/I$ factors  as illustrated in the diagram below
 \[\xymatrix{
     S/J \ar@{=}[r] & \PF_m(f) \ar[r]\ar[dr]_{\bar\ph} & \PF_{m:1}(\bar g) \ar[d]^\cong \ar@{=}[r] & R'/I' \\
                    &                                  & \PF_{m:1}(f)              \ar@{=}[r] & R/I
   }
  \]

 The horizontal arrow in the above diagram is the homomorphism  induced by the inclusion $S\hookrightarrow R'$ at the level of the respective quotients. Since a factorization of the factor $g$ of $f$ into $m$ degree 1 factors gives a factorization of $f$ having $m$ degree 1 factors, there is an isomorphism $R'/I'\cong R/I$ represented by the vertical arrow in the diagram. The commutativity of the diagram stems from the fact that, in $\PF_{m:1}(\bar g)$, $c_i$ is equal to the  the $i$-th symmetric polynomial in the variables $c_{11},\ldots,c_{m1}$.

  Applying Lemma~\ref{lemma:case1} to $\PF_{m:1}(f)$ and to $\PF_{m:1}(\bar g)$ respectively, we
  obtain the isomorphisms displayed below, where the first is an $A$-module isomorphism and the second is a $\PF_m(f)$-module isomorphism, hence also an $A$-module isomorphism.
  \begin{eqnarray*}
  \PF_{m:1}(f) &\cong& A^{n(n-1)\dots(n-m+1)}\\
  \PF_{m:1}(\bar g) &\cong& \PF_m(f)[c_{i1}\vert 1\leq i\leq m]/\langle \bar g\rem g_1\ldots g_m \equiv 0 \rangle \cong \PF_m(f)^{m(m-1)\dots 1}.
  \end{eqnarray*}
  
  Since  we have $ \PF_{m:1}(f) \cong \PF_{m:1}(\bar g)$ from the preceding diagram, the $A$-module isomorphism $A^{n(n-1)\dots(n-m+1)}\cong\PF_m(f)^{m(m-1)\dots 1}$ follows, 
and thus
  $\PF_m(f)$ is a finitely generated projective $A$-module having rank 
  $$N =(n(n-1)\dots(n-m+1)) / (m(m-1)\dots 1) = {\binom n m}.$$
Observe that $\ph$ is
  monic (see Definition \ref{definition:monicmap}), because $\l_{\rm{grevlex}}(\ph(c_{1i})) = c_{11}
  \dots c_{i1}$ is monic and the monomials of the form $c_{11} \dots c_{i1}$
  are multiplicatively independent.  Observe also that the composite map $\bar \ph: S/J \to R/I$ is injective by applying Lemma \ref{lemma:case1} to see that $S/J \to R'/I'$ is injective. Thus we conclude that  $J = \ph^{-1}(I)$.

  With an eye toward applying Proposition \ref{proposition:cond}, we consider for each sequence $\alpha$
 with  $n-m \ge \alpha_1 \ge \dots \ge \alpha_m \ge 0$ the symmetric sum $s_\alpha=\sum_{\sigma\in S_n}\prod_{i=1}^m c_{1\sigma(i)}^{\alpha_i}$ of the
  distinct monomials obtained from $c_{11}^{\alpha_1}\dots c_{m1}^{\alpha_m}$
  by permuting the variables in all possible ways. We define  $L$
  to be the $A$-submodule of $R$ generated by the set $\{s_\alpha\}$. Since $\ph$ sends the variables of $S$ to the elementary symmetric polynomials in $R'$, the theorem on
  elementary symmetric functions gives that the image of $\ph$ consists of all
  the symmetric polynomials in $R$, so $L \subseteq \ph(S)$, establishing hypothesis
  \ref{cond:phsb} of Proposition \ref{proposition:cond}.  Moreover, according
  to our computation in Lemma \ref{lemma:case1}, $\l_{\rm{grevlex}}(I)=\langle c_{11}^{n-m+1}, c_{21}^{n-m+2},
  \dots, c_{m1}^n \rangle$. Since the monomials appearing in $s_\alpha$ are all smaller in the graded reverse lexicographic order than each generator of $\l_{\rm{grevlex}}(I)$, it follows that each $s_{\alpha}$ is already
  reduced modulo $I$, establishing hypothesis \ref{cond:redI} of 
  Proposition \ref{proposition:cond}.  By a standard combinatorial argument, the number of sequences
  $\alpha$ referred to above is ${\binom n m} = N$. Since the generators of $L$
  have distinct lead terms, we see that $L$ is a free $A$-submodule and summand
  of $R$ whose rank is $N$, establishing hypotheses \ref{cond:sumd} and
  \ref{cond:rank} of  Proposition \ref{proposition:cond}.  Therefore the Proposition applies, showing
  that $J$ is monic and that $S/J$ is a free $A$-module of rank $N$, proving claims
  \ref{meatyprop:monic} and \ref{meatyprop:rank} of Proposition \ref{proposition:meatyprop}.

  Now we describe the monomial ordering on $S$ induced by $\ph$.  By Lemma
  \ref{lemma:case1}, we are free to choose any monomial ordering on $R$ such
  that $c_{11} > c_{21} > \dots > c_{m1}$.  We choose the graded reverse
  lexicographic ordering.  Then we have the following equivalent statements
  \begin{align*}
    c^\beta & \le c^\gamma  \\
    & \Leftrightarrow \l(\ph(c^\beta)) \le_{\rm{grevlex}} \l(\ph(c^\gamma)) 
    \\
     & \Leftrightarrow \prod_{i=1}^m\l(\ph(c_{i1}))^{\beta_i} \le_{\rm{grevlex}}  \prod_{i=1}^m\l(\ph(c_{i1}))^{\gamma_i }
\\
   & \Leftrightarrow \prod_{i=1}^m\l\left(\sum_{1\leq\alpha_1\leq\ldots\alpha_i\leq m} \prod_{j=1}^{i}c_{\alpha_j1}\right)^{\beta_i} \le_{\rm{grevlex}}  \prod_{i=1}^m\l\left(\sum_{1\leq\alpha_1\leq\ldots\alpha_i\leq m} \prod_{j=1}^{i}c_{\alpha_j1}\right)^{\gamma_i} 
\\
   & \Leftrightarrow c_{11}^{\beta_1}(c_{11}c_{21})^{\beta_2}\dots(c_{11}\dots c_{m1})^{\beta_m} \le_{\rm{grevlex}}  c_{11}^{\gamma_1}(c_{11}c_{21})^{\gamma_2}\dots(c_{11}\dots
    c_{m1})^{\gamma_m}\\
   & \Leftrightarrow c_{11}^{\beta_1 + \dots + \beta_m} c_{21}^{\beta_2 + \dots + \beta_m} \dots
    c_{m1}^{\beta_m} \le_{\rm{grevlex}}  c_{11}^{\gamma_1 + \dots + \gamma_m} c_{21}^{\gamma_2 + \dots + \gamma_m} \dots
    c_{m1}^{\gamma_m} 
    \\
   & \Leftrightarrow 
    \begin{cases} \beta_1 + 2 \beta_2 \dots + m \beta_m\le \gamma_1 + 2 \gamma_2 \dots + m \gamma_m \mbox{ and }\\
    (-\beta_m, - \beta_m - \beta_{m-1}, \dots, -\beta_1)\le_{\rm{lex}} (-\gamma_m, - \gamma_{m-1}, \dots, -\gamma_1)
    \end{cases}
     \end{align*}
  This monomial ordering on $S$ is the {\em graded reverse lexicographic}
  ordering, with each variable $c_{1i}$ taken to be of degree $i$ (the natural
  degree), and with the variables listed in the order $c_{11} > \dots > c_{1m}$.
Had we retained the original lexicographic ordering on the monomials of $R$,
  the ordering on $S$ induced by $\ph$ would have been the (less natural)
  graded reverse lexicographic ordering with the variables all taken to be of degree $1$
  and listed in the order $c_{m1} >\dots>c_{11}$.

  We can easily compute $\l_{\rm{grevlex}}(J)$ now.  By Proposition \ref{lemma:one} part \ref{one:monoms} we know that
  \begin{align*}
    c^\beta \in \l_{\rm{grevlex}}(J) 
    & \Leftrightarrow \l_{\rm{grevlex}}(\ph(c^\beta)) \in \l_{\rm{grevlex}}(I) \\
    & \Leftrightarrow c_{11}^{\beta_1 + \dots + \beta_m}
      c_{21}^{\beta_2 + \dots + \beta_m} \dots c_{m1}^{\beta_m} \in \l_{\rm{grevlex}}(I)
 \end{align*}
  A monomial $ c_{11}^{\beta_1 + \dots + \beta_m} c_{21}^{\beta_2 + \dots +
    \beta_m} \dots c_{m1}^{\beta_m}$ of $R$ lies in $\l_{\rm{grevlex}}(I)$, which was
  described above, if and only if one of the following inequalities is
  satisfied:
  \begin{enumerate}
    \item[] $n-m+1 \le \beta_1 + \dots + \beta_m$
    \item[] $n-m+2 \le \beta_2 + \dots + \beta_m$
    \item[] \dots
    \item[] $n \le \beta_m $
  \end{enumerate}
  The left sides of the inequalities are increasing, and the right sides are
  decreasing, so if any of the inequalities is satisfied, so is the first one.
  Since $c^\beta \in \l_{\rm{grevlex}}(J) \Leftrightarrow n-m+1 \le \beta_1 + \dots +
  \beta_m$, we see that $\l_{\rm{grevlex}}(J)=\{ c^\beta \mid \beta_1 + \dots
  + \beta_m = n-m+1 \}=\langle c_1,\cdots, c_m\rangle^{n-m+1} $ and $\red_J S$ has the set $\{ c^\beta \mid \beta_1 +
  \dots + \beta_m \le n-m \}$ as $A$-basis.  Thus we have proved part
  \ref{meatyprop:leadterms} of Proposition \ref{proposition:meatyprop}.

Parts \ref{meatyprop:monic} and \ref{meatyprop:rank} have been established in the course of the proof. 
Finally, part \ref{meatyprop:inject} follows because the rank in part \ref{meatyprop:rank} is nonzero.
\end{proof}

\section{Computation of pushforward maps}\label{section:computepilowerstar}

As in the introduction, we let $\EE$ be a vector bundle of rank $n$ on a
nonsingular variety $X$, and let $c(\EE) = 1 + c_1\EE + \dots + c_n \EE$ denote
the total Chern class of $\EE$ in the intersection ring $A(X)$.  Choose a
partition $n = n_1 + \cdots + n_{r}$ and let $\FF := \Fl_{n_1,\dots,n_r}(\EE)$
denote the flag bundle of $\EE$ that parametrizes filtrations of $E$ over points of $X$, where the successive
subquotients from the filtration have ranks $n_1, \dots, n_r$.  Let $\pi : \FF \to X$ denote
the structure map, and let $\pi^* \EE = \EE_r \supseteq \EE_{r-1} \supseteq
\dots \supseteq \EE_1 \supseteq \EE_0 = 0$ denote the tautological filtration,
where $\rank(\EE_i/\EE_{i-1}) = n_i$. 

There is a close connection between $A(X)$ and the splitting algebras defined in the previous sections. Specifically, there is a natural map $\Phi : \PF_{n_1,\dots,n_r}(c_x(\EE)) \to A(\FF)$ provided by the
  factorization $\pi^* c_x(\EE) = c_x(\EE_r/\EE_{r-1}) \cdot \dots \cdot
  c_x(\EE_2/\EE_1) \cdot c_x(\EE_1/\EE_0)$, which sends the variable $c_{ij}$
  to the Chern class $c_j(\EE_i/\EE_{i-1})$. The following theorem of Grothendieck describes the intersection ring of $\FF$ by means of this homomorphism.
 
\begin{theorem}[{\cite[Theorem 1, p.~4-19]{grothendieck-1958}}]\label{groththm}
  The natural map $\Phi : \PF_{n_1,\dots,n_r}(c_x(\EE)) \to A(\FF)$  is an isomorphism.
\end{theorem}

We are interested in computation of the push-forward map $\pi_* : A(\FF) \to
A(X)$, because, in the case where $X$ is a point, $ \pi_* (t)=\int t $ is a
number, often with enumerative significance.  Parameter varieties constructed
iteratively as a sequence of flag bundles over flag bundles can also be handled,
by composing the various maps $\pi_*$.

Taking into account that $n = n_1 + \cdots + n_{r}$ and applying Theorem
\ref{theorem:ourtheorem}, we see that the single monomial of highest degree in the
basis of $\PF_{n_1,\dots,n_r}(c_x(\EE))$ is $c_{2,n_2}^{n_1}
c_{3,n_3}^{n_1+n_2} \dots c_{r,n_r}^{n_1+\dots+n_{r-1}}$.  The corresponding
element of $A(\FF)$ is $$\eta := c_{n_2}(\EE_2/\EE_1)^{n_1}
c_{n_3}(\EE_3/\EE_2)^{n_1+n_2} \dots
c_{n_r}(\EE_r/\EE_{r-1})^{n_1+\dots+n_{r-1}}.$$ 

\begin{proposition}\label{proposition:integral}
  $\pi_*\eta = 1$
\end{proposition}

\begin{proof}
  The degrees of the elements in the claim are equal, so assuming $X$ is
  irreducible, there is an integer $m$ so that $\pi_*\eta = m$.  In order to
  prove that $m=1$, we may replace $X$ by a non-empty affine open subset where $\EE$ is trivial.  Let
  $\EE'$ be a trivial subbundle of $\EE$ of rank $n_1+\dots+n_{r-1}$, and
  consider the flag bundle $\FF' := \Fl_{n_1,\dots,n_{r-1}}(\EE')$, its
  structure map $\pi' : \FF' \to X$, and its tautological filtration $\pi'^*
  \EE' =: \EE'_{r-1} \supseteq \EE'_{r-2} \supseteq \dots \supseteq \EE'_1
  \supseteq \EE'_0 = 0$.

  Consider the map $\lambda : \FF' \to \FF$ over $X$ defined by the filtration
  $ \EE'_r := \pi'^* \EE \supseteq \pi'^* \EE' =: \EE'_{r-1} \supseteq
  \EE'_{r-2} \supseteq \dots \supseteq \EE'_1 \supseteq \EE'_0 = 0$ of $\pi'^*
  \EE$; that means that $\lambda$ pulls back the tautological filtration on
  $\FF$ to this filtration on $\FF'$.  It is a closed immersion whose image is
  defined by the vanishing of the composite map $ \pi^* \EE' \hookrightarrow
  \pi^* \EE \twoheadrightarrow \EE_r/\EE_{r-1} $, because on that zero locus,
  the containment $\pi^* \EE' \subseteq \EE_{r-1}$, where the subbundles have
  the same rank, forces $\pi^* \EE'= \EE_{r-1}$.  The codimension is
  $(n_1+\dots+n_{r-1}) n_r$, which equals the products of the ranks of the source and target of the
  composite map, so $\lambda(\FF')$ is the complete intersection of the zero
  loci of the corresponding global sections of $\EE_r/\EE_{r-1} $, each of
  which has cycle class equal to the top Chern class
  $c_{n_r}(\EE_r/\EE_{r-1})$, by \cite[Theorem~2, p.~153]{grothendieck-1958}.  It follows that $\lambda_* 1 =
  c_{n_r}(\EE_r/\EE_{r-1})^{n_1+\dots+n_{r-1}}$.

  Let $$\eta' := c_{n_2}(\EE'_2/\EE'_1)^{n_1} c_{n_3}(\EE'_3/\EE'_2)^{n_1+n_2}
  \dots c_{n_{r-1}}(\EE'_{r-1}/\EE'_{r-2})^{n_1+\dots+n_{r-2}} \in A(\FF').$$
  By induction we may assume that $\pi'_* \eta' = 1$.

  Letting $$\tilde \eta := c_{n_2}(\EE_2/\EE_1)^{n_1}
  c_{n_3}(\EE_3/\EE_2)^{n_1+n_2} \dots
  c_{n_{r-1}}(\EE_{r-1}/\EE_{r-2})^{n_1+\dots+n_{r-2}} \in A(\FF),$$ we see
  that $\lambda^* \tilde \eta = \eta'$.

  We compute $\pi_* \eta = \pi_*((\lambda_* 1) \tilde\eta) = \pi_* \lambda_*
  \lambda^* \tilde\eta = \pi'_* \eta' = 1$, proving the claim.
\end{proof}

A consequence of the proposition is that the computation of $\pi_*$ can be
accomplished by reducing an element to normal form modulo the \Gb{} and taking
the coefficient of $\eta$, as we now show.

\begin{corollary}\label{computepilowerstar}
  Let $\zeta$ be an element of $A(\FF)$.  Using the basis provided by part 3 of Theorem
  \ref{theorem:ourtheorem} write $\zeta$ in the form $\sum_\alpha \pi^*(z_\alpha)
  \Phi(c^\alpha)$, and choose $\beta$ so $\Phi(c^\beta) = \eta$.  Then the identity $\pi_* z
  = z_\beta$ holds.
\end{corollary}

\begin{proof}
  Observe that $\pi_* \Phi ( c^\alpha ) = 0$ for all $\alpha \ne \beta$, for
  dimension reasons.  We compute $\pi_* z = \pi_* (\sum_\alpha \pi^*(z_\alpha)
  \Phi(c^\alpha) = \sum_\alpha \pi_* ( \pi^*(z_\alpha) \Phi(c^\alpha) ) =
  \sum_\alpha z_\alpha \pi_* ( \Phi(c^\alpha) ) = z_\beta$.
\end{proof}

\section{Examples}

\begin{example}[Lines in $\PP^3$]
  For a first example, we show how the number of lines meeting 4 lines in space is computed using the algorithm developed above.  Let $\FF = \Fl_{2,2}(k^4)$ be the Grassmannian of
  lines in projective 3-space.  The intersection ring $A(\FF)$ of $\FF$ is
  isomorphic to
  $\PF_{2}(x^4)$.  Write $g(x) = x^2 + c_1 x + c_2$.  The equations defining
  $\PF_{2}(x^4)$ are the coefficients of $x^4 \quo g(x)$, and so $A(\FF) \cong \ZZ[c_1,c_2]/I$,
  where
  $I = (-c_1^3+2c_1c_2, -c_1^2c_2 + c_2^2)$.  In fact, by Proposition \ref{minors} or Corollary \ref{meatyprop:Gb}, a \Gb{} for $I$ is given by the maximal minors of the matrix
  $$\begin{bmatrix}
  1 & 1 & 0 & 0 \\
  0 & c_1 & 1 & 0\\
  0 & c_2 & c_1 & 1\\
  0 &  0 & c_2 & c_1 \\
  0 & 0 &  0 & c_2 
  \end{bmatrix},$$
  i.e., the polynomials $\{c_1^3-2c_1c_2, c_1^2c_2-c_2^2,  c_1c_2^2,  c_2^3\}$, where we have written leading terms first. Thus 
  reduction modulo the \Gb{} of $I$
  amounts to the following rewriting rules.
  \begin{align*}
    c_1^3 &\mapsto 2c_1c_2 \\
    c_1^2c_2 &\mapsto c_2^2 \\
    c_1c_2^2 &\mapsto 0 \\
    c_2^3 &\mapsto 0
  \end{align*}
  This is the \Gb{} for any monomial ordering such that $c_1 > c_2$.  In terms of
  geometry, $c_1$ is the class of the subset of (projective) lines meeting a
  given line, 
  and $c_2$
  is the class of the subset of lines lying in a plane.  The class $c_2^2$ is the
  class of a single line, i.e., a point in $\FF$, by Proposition \ref{proposition:integral}.

  Now let's compute the number (two) of lines meeting 4 general lines in space.
  It is $\int c_1^4$ and can be computed by taking the coefficient of $c_2^2$ in
  the normal form of $c_1^4$, as follows.
  $$\int c_1^4 = \int 2c_1^2c_2 = \int 2c_2^2 = 2.$$

  Alternatively, we could split $g$ into two linear forms, as in Section
  \ref{section:another-proof}.  Geometrically, we are moving to a flag bundle over $\FF$.  Write
  $x^2 + c_1 x + c_2 = (x + t)(x + u)$ and postulate that it divides $x^4$.
 Since $c_2=tu$, the class of a point is $c_2^2 = t^2 u^2$. We now construct the \Gb{} by composition as in our first proof of Proposition \ref{proposition:meatyprop} in section 2. First we construct 
 \[B_1=\PF_{1}(f)=A[u]/\langle f\rem (x+u)\equiv 0\rangle.
 \] A \Gb{} for the ideal defining this quotient is given by Lemma \ref{lemma:onelinear} as $\{f(-u)=u^4\}$. Next we construct a \Gb{} for the ideal defining 
 \[B_2=B_1[t]/\langle f(x)//(x+u) \rem (x+t) \equiv 0\rangle.\]
  Setting $q(x):=f(x)//(x+u)=x^3-ux^2+u^2x-u^3\in B_1$, a \Gb{} for the ideal $\langle f(x)//(x+u) \rem (x+t) \equiv 0\rangle$ is given by $q(-t)=-t^3-ut^2-u^2t-u^3$, again by Lemma \ref{lemma:onelinear}. Thus by composition a monic \Gb{} for the ideal defining $\PF_2(f)$ as a quotient of $A[u,t]$ is given by $\{t^3+ut^2+u^2t+u^3,u^4\}$. 
Reduction modulo the resulting \Gb{} given by the second proof of Proposition \ref{proposition:meatyprop}  given in section 4 of this paper amounts to the following rewriting rules.
  \begin{align*}
    u^4 &\mapsto 0 \\
    t^3 &\mapsto - t^2 u - t u^2 - u^3
  \end{align*}
  We compute the desired number as follows.
  \begin{align*}
        \int c_1^4
    & = \int (t+u)^4
      = \int (t^3 + 3 t^2 u + 3 t u^2 + u^3)(t+u)
      = \int (2 t^2 u + 2 t u^2) (t + u) \\
    & = \int 2 t u (t+u)^2 
      = \int 2 (t^3 + 2 t^2 u + t u^2) u 
      = \int 2 ( t^2 u - u^3 ) u \\
    & = \int 2 ( t^2 u^2 - u^4 ) 
      = 2 \int t^2 u^2
      = 2.
  \end{align*}
\end{example}

\begin{example}\label{gr6:10}
Suppose that $f \in A[x]$ is monic of degree $n$, and suppose that $1 \leq m
\leq n-1$.  Since $\PF_{m}(f)$ and $\PF_{n-m}(f)$ are isomorphic, one could ask
which one is the best to use in computations.  For example, how large are the
corresponding \Gbs?  We compute the following examples using {\em Macaulay2}.

If $n = 10$, and $m = 6$, then the \Gb{} for the ideal (in 6 variables) defining
$\PF_{6}(x^{10})$ has 252 generators (in 6 variables $c_1, \dots, c_6$), and
$\l(I) = (c_1, \dots, c_6)^5$.  The total number of monomials in an auto-reduced
\Gb{} of $I$ is 1458.  The maximum number of monomials occuring in a \Gb{} element is
18. 

By contrast, the \Gb{} for the ideal (in 4 variables) defining the isomorphic
ring $\PF_{4}(x^{10})$ has 120 generators (in 4 variables $c_1, \dots, c_4$),
and $\l(I) = (c_1, \dots, c_4)^7$.  The total number of monomials in an auto-reduced
\Gb{} of $I$ is 773.  The maximum number of monomials occuring in a \Gb{} element is
17.
\end{example}

\begin{example}
  In 2009, at MSRI, O. Debarre and C. Voisin asked us the following question
  \cite{debarre:voisin}.  Let $\FF = \Fl_{6,4}(k^{10})$ be the Grassmannian of
  6-planes in a 10-dimensional vector space over a field $k$.  Let $S$ be the
  tautological subbundle of $\OO_\FF^{10}$, and let $Z \subseteq \FF$ be a general
  section of the vector bundle $\Lambda^3 S^*$.  As this is a rank 20 bundle,
  and $\dim \FF = 24$, $Z \subseteq \FF$ has dimension 4.  As it turns out, $Z$ is
  also non-singular, and $\OO_Z$ is resolved by the Koszul complex with terms
  $\Lambda^i E$.  The question they asked was to compute $\chi (\OO_Z) = \Sigma
  (-1)^i \chi (\Lambda^i E).$

  The intersection ring $A(\FF)$ is the one considered in Example \ref{gr6:10}.
  Computing with the {\em Macaulay2} package {\em Schubert2}, using the \Gb{}
  defining $A(\FF)$, one finds that $\chi(\OO_Z) = 3$.  The computation proceeds
  by using the Hirzebruch-Riemann-Roch Theorem to convert the computation of the
  Euler characteristic to the computation of an integral, which in turn is
  accomplished using the algorithm described in this paper.

  Debarre and Voisin were then able to deduce that $Z$ defines an irreducible
  symplectic manifold, for details, see \cite{debarre:voisin}. Basically, by
  adjunction, one knows that $K_Z = 0$, and then by a theorem of
  Beauville-Bogomolov, $Z$ has a finite \'etale cover $Y$, of some degree $m$,
  which is a product of Abelian varieties (having $\chi(\OO) = 0$), Calabi-Yau
  manifolds ($\chi(\OO) = 2$), and irreducible symplectic varieties ($\chi(\OO)
  = 3$).  Since $\chi(\OO_Y) = m \chi(\OO_Z) = 3m$, one deduces eventually that
  $m=1$, and that $Z$ is a new example of an irreducible symplectic $4$-fold.
\end{example}

\section{The isotropic flag bundle}\label{section:isotropic-flag-bundle}

In this section we prove an analogue of Proposition \ref{theorem:ourtheorem}
for the intersection ring of isotropic flag bundles (and Grassmannians),
thereby enabling convenient computations of intersection theory on them.

Suppose that $\EE$ is a vector bundle of rank $2n$ on a nonsingular variety $X$
equipped with a nonsingular alternating (symplectic) form, i.e., it satisfies
the identity $\langle e, e \rangle = 0$ for any section $e$, implying also that
$\langle e, e' \rangle = - \langle e', e \rangle$.  The isomorphism $\EE \cong
\EE^*$ gives an equation of total Chern classes $c(\EE) = c^*(\EE)$, where
$c^*(\EE) := c(\EE^*) = \sum_i (-1)^i c_i(\EE)$.  That implies that $c_i(\EE)$
is $2$-torsion when $i$ is odd, but more is true, as the following lemma shows.

Let $\LL$ be a subbundle of $\EE$,  i.e., the quotient sheaf $\EE/\LL$ is locally free.
   Restriction to $\LL$ gives a surjective map $\EE^* \to \LL^*$ because the inclusion
$\LL \to \EE$ is locally split. Since $X$ is nonsingular, there is an isomorphism $\EE\cong \EE^*$. We define  $\LL^\perp$ to be the kernel of the composite map $\EE \to \LL^*$, i.e., the subbundle consisting of the vectors in the orthogonal complement of $\LL$ with respect to the alternating form.  It is a subbundle because the restriction map utilized in this construction is surjective.  

\begin{lemma}
  Suppose that $\EE$ is a vector bundle on a nonsingular variety $X$ equipped
  with a nonsingular alternating form.  Then $c_i(\EE) = 0$ if $i$ is odd.
\end{lemma}

\begin{proof}
  Let $\rho : \PP(\EE^*) \to X$ be the structure map for $\PP(\EE^*)$.  Consider
  the tautological line bundle $\LL \subseteq \rho^* \EE$.  Since the form is
  alternating, $\LL \subseteq \LL^\perp$.  Taking into account the isomorphism
  $\rho^* \EE / \LL^\perp \cong \LL^*$ and using multiplicativity of the total Chern
  class, we see that $\rho^* c(\EE) = c(\LL) c^*(\LL) c(\LL^\perp/\LL)$.  Since
  $\LL^\perp/\LL$ inherits a nonsingular alternating form, we may assume, by
  induction on the rank of $\EE$, that its total Chern class vanishes in odd
  degree.  Since $c(\LL) c^*(\LL)$ also vanishes in odd degree, so does $\rho^*
  c(\EE)$, and hence so does $c(\EE)$, by injectivity of $\rho^*$.
\end{proof}

Given $r, n_1, \dots, n_r$ consider the {\em isotropic flag bundle} $\FF =
\Fl^\iso_{n_1,\dots,n_r}(\EE)$ over $X$ which parametrizes the filtrations of
$\EE$ over points of $X$ by totally isotropic subbundles and their annihilators, where the successive
quotients from the filtration have ranks $n_r, \dots, n_1, 2(n - n_1 - \dots - n_r), n_1, \dots,
n_r$.  Let $\pi_\FF$ denote the structure map $ \FF \rightarrow X $.  The
tautological filtration $\calF(\FF)$ on $\FF$ has the form 
$$\calF(\FF) := \left
( \pi_\FF^* \EE = \calF^r(\FF) \supseteq \dots \supseteq \calF^0(\FF) \supseteq
\calF_0(\FF) \supseteq \dots \supseteq \calF_r(\FF) = 0 \right ),$$ 
where each
$\calF_i(\FF) := \calF^i(\FF)^\perp$ is totally isotropic and $\rank
(\calF^i(\FF) / \calF^{i-1}(\FF)) = r_i$, for each $i$.  Each $\calF^i(\FF)$ is
     {\em totally coisotropic}, in the sense that its annihilator
     $\calF_i(\FF)$ is totally isotropic.

We introduce the following additional notation.
\begin{align*}
    \calF^i & := \calF^i(\FF) \\
    \calF_i & := \calF_i(\FF) \\
    \calF^{i,j} := \calF^{i,j}(\FF) & := \calF^{i}(\FF)/\calF^{j}(\FF) \\
    \calF_{j,i} := \calF_{j,i}(\FF) & := \calF_{j}(\FF)/\calF_{i}(\FF) \\
    \calF^i_j := \calF^i_j(\FF) & := \calF^{i}(\FF)/\calF_{j}(\FF) 
\end{align*}

Additivity of total Chern classes and the isomorphisms $\calF^{i,i-1}
\cong \calF_{i-1,i}^*$ induced by the alternating form provide the
following identity of total Chern classes.
\[
 \pi_\FF^*c(\EE) = c(\calF^{r,r-1})\cdot\ldots\cdot c(\calF^{1,0})
 \cdot c(\calF^0_0) \cdot c^*(\calF^{1,0})\cdot\ldots\cdot
 c^*(\calF^{r,r-1})
\]

In order to model such factorizations abstractly, we let $A$ be a commutative
ring and consider a monic polynomial $f(x) = x^{2n} + b_2 x^{2n-2} + \dots +
b_{2n-2} x^2 + b_{2n} \in A[x]$ of degree $2n$ with no terms of odd degree,
playing the role of $c_x(\EE)$.

For a monic polynomial $g(x)$ of degree $m$, we define $g^*(x) := (-1)^m
g(-x)$, so $g^*(x) = c_x(\DD^*)$ if $g(x) = c_x(\DD)$ for a vector bundle
$\DD$ of rank $m$.

We consider divisibility of $f(x)$ by a product $g_r(x) \dots g_1(x) g_1^*(x)
\dots g_r^*(x)$, where $g_i(x) = x^{n_i} + c_{i,1} x^{n_i-1} + \dots +
c_{i,n_i-1} x + c_{i,n_i}$, and where $g_i(x)$ plays the role of
$c_x(\calF^{i,i-1})$.  The universal example of such a
tuple $(g_1,\dots,g_r)$ is the one whose coefficients are in the quotient ring
$\PF^{\rm iso}_{n_1,\dots,n_r}(f) := S/J$, where $$S =
A[c_{1,1},\dots,c_{1,n_1};\dots;c_{r,1},\dots,c_{r,n_r}]$$ and $$J = \langle
f(x) \rem g_1(x) \dots g_r(x) g_r^*(x) \dots g_1^*(x) \equiv 0 \rangle.$$ We
let $c_{i,j}$ have degree $j$ and order the monomials of $S$ with the block
ordering $[c_{1,1},\dots,c_{1,n_1}] \gg \dots \gg [c_{r,1},\dots,c_{r,n_r}]$
arising from the graded reverse lex ordering with $c_{i,1} > \dots > c_{i,n_i}$
on each block.

\begin{theorem}\label{theorem:isotropicgroth}
  The map $ \ph : \PF^\iso_{n_1,\dots,n_r}(c_x(\EE)) \to A(\Fl^\iso_{n_1,\dots,n_r}(\EE))$ arising from universality and the factorization
  of total Chern classes above, sending $g_i(x)$ to
  $c_x(\calF^{i,i-1})$, is an isomorphism.
\end{theorem}

\begin{proof}
  We use Grothendieck's method \cite[Theorem 1, p.~4-19]{grothendieck-1958} for
  the Grassmannian, adapted to the isotropic case.  The idea is to construct a
  proof by induction via slight refinements of the tuples $(n_1,\dots,n_r)$
  that correspond to projective bundles on the geometry side and to ring
  extensions defined by one monic relation in one new variable on the algebra
  side.  A refinement of type A will involve
  replacing ${\bf n}$ by ${\bf n'} =
  (n_1,\dots,n_{i-1},n_i-1,1,n_{i+1},\dots,n_r)$,
  and one of type B replaces ${\bf n} =
  (n_1,\dots,n_r)$ by ${\bf n'} = (1,n_1,\dots,n_r)$.
  In each case, the forgetful
  maps give a commutative square
  \[\xymatrix{
    \PF^\iso_{\bf n'}(c_x(\EE)) \ar[r]^{\ph'} & A(\Fl^\iso_{\bf n'}(\EE)) \\
    \PF^\iso_{\bf n}(c_x(\EE)) \ar[u] \ar[r]^\ph & A(\Fl^\iso_{\bf n}(\EE)) \ar[u]
   }
  \]

  Consider first a refinement of type A.  On the left side the extension is
  defined by adjoining a root of the monic polynomial $g_i(x)$, which provides
  a linear factor and a complementary factor of degree $n_i-1$.  On the right
  side, the bundle $ \pi' : \Fl^\iso_{\bf n'}(\EE) \to \Fl^\iso_{\bf n}(\EE)$
  is isomorphic to the projective bundle of $\calF^{i,i-1}(\Fl^\iso_{\bf
    n}(\EE))$, which filters the pull-back of $\calF^{i,i-1}(\Fl^\iso_{\bf
    n}(\EE))$ by subquotient bundles of ranks $n_i-1, 1$.

  Consider next a refinement of type B.  On the left side the extension is
  defined by adjoining a root $\alpha$ of the monic polynomial $$h(x) := f(x)
  \quo g_1(x) \dots g_r(x) g_r^*(x) \dots g_1^*(x),$$ which provides a linear
  factor $g_0(x) = x - \alpha$ of $h(x)$.  Since $h(x)$ has no terms of odd
  degree, we may think of it as a polynomial in  $x^2$ with a
  ``root'' $\alpha^2$ and conclude that it has a factor $x^2 -
  \alpha^2 = (x - \alpha) (x + \alpha) = g_0(x) g_0^*(x)$, as needed.  On the
  right side, the bundle $\Fl^\iso_{\bf n'}(\EE) \to \Fl^\iso_{\bf n}(\EE)$ is
  isomorphic to the projective bundle $ \PP(\calF^0_0(\Fl^\iso_{\bf n}(\EE)))
  \to \Fl^\iso_{\bf n}(\EE)$, because any rank $1$ subbundle of a vector bundle
  equipped with an alternating form is totally isotropic.

  It follows from Grothendieck's theorem \ref{groththm} that for a projective
  bundle $\PP(\DD) \to X$ associated to a vector bundle $\DD$, the ring
  extension $A(X) \to A(\PP(\DD))$ is defined by adjoining a root $\alpha := -
  c_1(\LL)$ of the monic polynomial $c_x(\DD)$, where $\LL$ is the tautological
  line bundle.  The roots and equations on the left and right sides of the
  commutative square above correspond, so $\ph'$, as a map of abelian groups,
  is isomorphic to a direct sum of copies of $\ph$.  We conclude that $\ph'$ is
  an isomorphism if and only if $\ph$ is.

  Now we construct our proof by induction, starting from the base case where
  ${\bf n}$ is empty.  Refine it $n_1 + \dots + n_r$ times by refinements
  of type A to reach $(1,1,\dots,1)$, which can also be reached from
  $(n_1,\dots,n_r)$ by several refinements of type B, finishing the proof.
\end{proof}

\begin{definition}
  Let $I$ be a monomial ideal in a polynomial ring.  We let $\Sq I$ denote the
  monomial ideal generated by the squares of the monomials in $I$.
\end{definition}

Now we prove the analogue of Proposition \ref{proposition:meatyprop}.

\begin{proposition}\label{proposition:meatyprop2}
  Let $f(x) = x^{2n} + b_2 x^{2n-2} + \dots + b_{2n-2} x^2 + b_{2n} \in A[x]$
  be a monic polynomial of degree $2n$ with no terms of odd degree.  Take $m
  \in \NN$ with $m \le n$, and let $S/J = \PF_m^\iso(f)$.  Then 
  $\l(J) = \Sq ( \langle c_{11}, \dots, c_{1m} \rangle^{n-m+1} )$.
\end{proposition}

\begin{proof}
  We introduce the notation $g(x) := g_1(x)$ and $c_j := c_{1j}$ for each $j$.

  Letting $f^{(2)}(x) = x^{n} + b_2 x^{n-1} + \dots + b_{2n-2} x + b_{2n} \in
  A[x]$, we see that $f(x) = f^{(2)}(x^2)$.  Let $A[d_1,\dots,d_m]/I =
  \PF_m(f^{(2)})$, so that $\deg d_j = j$ and $h(x) := x^{m} + d_1 x^{m-1} +
  \dots + d_m$ is the universal monic divisor of $f^{(2)}$ of degree $m$.
  Observe that the coefficients of odd powers of $x$ in $g(x)g^*(x)$ vanish.
  Postulating that $g(x)g^*(x) = h(x^2)$ defines a map $\ph : A[d] \to S$ of
  $A$-algebras with $\ph(d_j) = \pm ( \dots - c_{j-1}c_{j+1} + c_j^2 -
  c_{j+1}c_{j-1} + \dots) = \pm ( c_j^2 - 2 c_{j-1}c_{j+1} + 2 c_{j-2}c_{j+2} +
  \dots)$ and with $\ph(I)\cdot S = J$.  The map $\ph$ doubles degrees, so if
  we introduce the notation $\degreemult S k$ to denote the operation on a
  graded ring $S$ that multiplies all the degrees by $k$, then we may regard
  $\ph$ as a homogeneous map $\ph : \degreemult{A[d]}{2} \to S$.  Since the
  inequality $c_j^2 > c_{j-t}c_{j+t}$ holds in the reverse lexicographic
  ordering (for $t > 0$), we see that $\tph(d_{j}) = \l(\ph(d_{j})) = c_j^2$,
  and thus $\ph$ is monic in the sense of Definition \ref{definition:monicmap}.
  Proposition \ref{proposition:meatyprop} provides a \Gb{} $G$ of $I$ whose lead terms are
  the monomials $d^\beta$ with $\beta_1 + \dots + \beta_m = n-m+1$.  Let $G' =
  \ph(G)$; its lead terms are the (square) monomials $c^{2\beta}$, with the
  same set of exponents $\beta$.  It is enough to show that $G'$ is a \Gb{} of
  $J$.

  By Lemma \ref{lemma:specialization} we may pass to the universal situation by
  replacing $S$ by the ring $\ZZ[b_2,\dots,b_{2n},c_1,\dots,c_m]$, promoting
  the ring elements $b_{2i}$ to variables.  To ensure that the lemma applies,
  as monomial ordering we take the block ordering $[c] \gg [b]$ that restricts
  to the graded reverse lexicographic ordering on $[b]$ and on $[c]$ (see
  Definition \ref{prodorddef}).  Assigning degrees $\deg b_{2j} = 2j$ and $\deg
  c_j = j$, we see that $J$ and the elements of $G$ are homogeneous.
  Similarly, we may replace $A[d]$ by the ring
  $\ZZ[b_2,\dots,b_{2n},d_1,\dots,d_m]$, but this time we take $\deg b_{2j} =
  j$ so that $I$ becomes homogeneous in it.  Moreover, the map $\ph$ remains
  homogeneous, for it sends $b_{2j}$ of degree $j$ in $\ZZ[b,d]$ to $b_{2j}$ of
  degree $2j$ in $\ZZ[b,c]$.

  By Lemma \ref{lemma:gbfields} we may replace $\ZZ$ by a field $F$.  In this
  situation, a finitely generated graded $S$-module $M$ has a Hilbert series.
  We will show that $\l(J) = \langle \l(G') \rangle$ by comparing the Hilbert
  series of the two ideals.

  The ideal $J$ is generated by the coefficients of $f(x) \rem g(x) g^*(x)$,
  which is homogeneous of degree $2n$ when $\deg x = 1$, so the coefficients of
  $x^{2m-2},x^{2m-4},\dots,x^2,1$ have complementary degrees $2n-2m+2, 2n-2m+4,
  \dots, 2n-2, 2n$.  The low-order coefficients $b_{2n-2m+2}, b_{2n-2m+4},
  \dots, b_{2n-2}, b_{2n}$ of $f(x)$ don't participate in the division, so each
  one appears as an isolated term in the equation of the same degree arising from $f\%g\equiv 0$ and in no other equations.  Eliminating those variables eliminates all the equations,
  too, so we see that $S/J \cong F[b_{2},\dots,b_{2n-2m},c_1,\dots,c_m]$,
  allowing its Hilbert series to be determined: $$H_F(S/J,T) = \frac{1}
  {(1-T^{2n-2m}) \dots (1-T^{2})\cdot(1-T^m) \dots (1-T^1)}.$$ From the observation
  in Definition \ref{definition:redI} that $\red_J S$ is the free graded
  $A$-submodule of $S$ generated by the monomials not contained in $\l(J)$, we
  see that $H_F(S/J,T) = H_F(S/\l(J),T)$.

  The map $\ph$ is monic, so $H_F(\degreemult {PF_m(f^{(2)})} 2,T) = H_F(PF_m(f^{(2)}),T^2)$
  tallies the distinct square monomials $c^{2\beta}$ not in
  $\langle\l(G')\rangle$.  Since $\l(G')$ consists just of square monomials,
  the additional monomials not in $\langle\l(G')\rangle$ can be obtained by
  multiplying each of the square monomials $c^{2\beta}$ by each of the
  square-free monomials $c^\delta$, which are tallied by $(1+T^m)\dots (1+T)$, yielding:
  \begin{align*}
      H_F(S/\langle\l(G')\rangle,T)
      & = H_F(PF_m(f^{(2)}),T^2)  \cdot (1+T^m) \dots (1+T)
  \end{align*}
  The identity 
  \begin{align*}
      H_F(PF_m(f^{(2)}),T) & = \frac 1 {(1-T^{2n-2m}) \dots (1-T^2)\cdot(1-T^{2m}) \dots (1-T^2)}
  \end{align*}
  was established in the proof of Proposition \ref{proposition:meatyprop}, after reductions
  analogous to the ones above were made to get the ring $PF_m(f^{(2)})$ to be
  graded.  Using that we get:
  \begin{align*}
      H_F(S/\langle\l(G')\rangle,T)
      & = \frac{(1+T^m) \dots (1+T)} {(1-T^{2n-2m}) \dots (1-T^2)\cdot(1-T^{2m}) \dots (1-T^2)} \\
      & = \frac
          {1}
          {(1-T^{2n-2m}) \dots (1-T^2)\cdot(1-T^{m})\dots(1-T^{1})} \\
      & = H_F(S/J,T) = H_F(S/\l(J),T)
  \end{align*}
  It follows from additivity over short exact sequences that the ideals
  $\langle\l(G')\rangle \subseteq \l(J)$ have the same Hilbert series over the
  field $F$, and thus they are equal, establishing the proposition.
\end{proof}

\begin{remark}
  As in Corollary \ref{meatyprop:Gb}, the ideal $J$
  has a minimal \Gb{} consisting of the maximal minors of a matrix.  For instance,
  the matrix has the following form if $n-m=4$, where $d_i$ is a quadratic polynomial
  in $c_1, \ldots, c_m$, with lead term $c_i^2$.
  \[\begin{pmatrix}
             1  &  1 &   0 &   0&   0&   0  \\
             b_2&d_1 &   1 &   0&   0&   0  \\
             b_4&d_2 & d_1 &   1&   0&   0  \\
             b_6&d_3 & d_2 & d_1&   1&   0  \\
             b_8&d_4 & d_3 & d_2& d_1&   1  \\
             b_{10}&d_5 & d_4 & d_3& d_2& d_1  \\
             b_{12}&d_6 & d_5 & d_4& d_3& d_2  \\
             b_{14}&d_7 & d_6 & d_5& d_4& d_3  \\
   \vdots&\vdots&\vdots&\vdots&\vdots&\vdots\\
         b_{2n-2}&  0 &   0 &   0& d_m& d_{m-1}  \\
             b_{2n} &  0 &   0 &   0&   0& d_m  \\
        \end{pmatrix}.
  \]
  As in the flag bundle case, this \Gb{} is often not auto-reduced.
\end{remark}

Now we consider the general universal factorization and prove the isotropic
analogue of theorem \ref{theorem:ourtheorem}.  We omit its analogous proof.

\begin{theorem}\label{theorem:ourtheorem2}
  Given a monic polynomial $f(x) = x^{2n} + b_2 x^{2n-2} + \dots + b_{2n-2} x^2
  + b_{2n} \in A[x]$ given $r \in \NN$ and given $n_1, \dots, n_r \in \NN$ with
  $n_1 + \dots + n_r \le n$, we let $S/J = \PF_{n_1,\dots,n_r}(f)$.  Then the
  lead term ideal $\l(J)$ is
  \begin{align*}
    \l( J )   = {}& \phantom{{}+{}} \Sq ( \langle  c_{11}, \dots, c_{1n_1} \rangle{}^{n - n_r - \dots - n_1 + 1} )  \\
     & + \Sq (\langle c_{21}, \dots, c_{2n_2}\rangle{}^{n - n_r - \dots - n_2 + 1}) \\
     & + \dots \\
     & + \Sq (\langle c_{r1}, \dots, c_{rn_r}\rangle{}^{n - n_r + 1}).
  \end{align*}
\end{theorem}

The following proposition is the analogue of Proposition
\ref{proposition:integral}, allowing the computation of $\pi_*$ by an analogue
of Corollary \ref{computepilowerstar}.

\begin{proposition}\label{proposition:integral2}
  In the context of theorem \ref{theorem:ourtheorem2}, the monomial of highest
  degree not in $J$ is
  \begin{align*}
  & \phantom{{}\cdot{}} (c_{1,1} \dots c_{1,n_1-1} c_{1,n_1}^{2n - 2n_r - \dots - 2n_1 + 1})\\
  & \cdot (c_{2,1} \dots c_{2,n_2-1} c_{2,n_2}^{2n - 2n_r - \dots - 2n_2 + 1}) \\
  & \cdot\dots\\
  & \cdot (c_{r,1} \dots c_{r,n_r-1} c_{r,n_r}^{2n - 2n_r + 1}).
  \end{align*}
  The isomorphism $\ph$ of theorem \ref{theorem:isotropicgroth} sends it
  to an element $\eta \in A(\FF)$ that satisfies the equation $$\pi_*\eta = 1$$ in $A(X)$.
\end{proposition}

  First we need an elementary lemma about polynomials.

  \begin{lemma}\label{lemma:JacobiTrudi}
    Suppose $R$ is a commutative ring, $p$ and $q$ are natural numbers,
    and $c = 1 + c_1 X + \dots + c_p X^p$ and $d = 1 + d_1 X + \dots + d_q X^q$ are
    polynomials over $R$.  If $c d = 1$, then $(-1)^{p+q} c_p^q = d_q^p$.
  \end{lemma}

  \begin{proof}
    We use the idea of the proof of the Jacobi-Trudi formula
    \cite[p.~10]{macdonald-MR1488699}, of which the lemma is a special case,
    corresponding to the rectangular conjugate partitions $p+p+\dots+p = q p =
    p q = q + \dots +q$; the method used here can be used to clarify the
    proof of the general case.  Let $V$ be the free $R$-module $R[X]/X^{p+q}$.
    Multiplication by $c$ and $d$ provides inverse linear maps $C, D : V \to V$
    of determinant $1$ whose matrices with respect to the basis $1, X, \dots,
    X^{p+q-1}$ are lower triangular with $1$'s on the diagonal.  We get the
    result from Laplace expansion by taking advantage of a pair of
    complementary upper triangular minors and calculating in $\det V = \Lambda
    ^{p+q} V$, which is a free module with basis $\{ X^0 \wedge \dots \wedge
    X^{p+q-1} \}$.
    \begin{align*}
      & \phantom{{}={}}
          (-1)^{p+q} c_p^q \cdot ( X^0 \wedge \dots \wedge X^{p+q-1} ) \\
      & = (-1)^{p+q} \cdot \left ( X^0 \wedge \dots \wedge X^{p-1} \wedge (\dots + c_p X^p) \wedge \dots \wedge (\dots + c_p X^{p+q-1}) \right ) \\
      & = (-1)^{p+q} \cdot \left ( X^0 \wedge \dots \wedge X^{p-1} \wedge c X^0 \wedge \dots \wedge c X^{q-1} \right ) \\
      & = c X^0 \wedge \dots \wedge c X^{q-1} \wedge X^0 \wedge \dots \wedge X^{p-1} \\
      & = c X^0 \wedge \dots \wedge c X^{q-1} \wedge c d X^0 \wedge \dots \wedge c d X^{p-1} \\
      & = (\det C) \cdot (X^0 \wedge \dots \wedge X^{q-1} \wedge d X^0 \wedge \dots \wedge d X^{p-1}) \\
      & = X^0 \wedge \dots \wedge X^{q-1} \wedge ( \dots + d_q X^q) \wedge \dots \wedge ( \dots + d_q X^{p+q-1}) \\
      & = d_q^p \cdot ( X^0 \wedge \dots \wedge X^{p+q-1} )
    \end{align*}
  \end{proof}

  We begin the proof of Proposition \ref{proposition:integral2}, following the same lines as in the
  proof of Proposition \ref{proposition:integral}.

  \begin{proof}[Proof of Proposition \ref{proposition:integral2}] The degrees of the elements in the claim are equal, so assuming $X$ is
  irreducible, there is an integer $m$ so that $\pi_*\eta = m$.  In order to
  prove $m=1$ we may replace $X$ by a non-empty affine open subset where $\EE$
  is trivial and has a complete chain of trivial totally isotropic or
  coisotropic subbundles of all intermediate ranks, from which we will draw all
  needed trivial subbundles.  Let $\EE'$ be a trivial totally coisotropic
  subbundle of $\EE$ of rank $2n-n_r$, consider $\EE'/\EE'^\perp$ with its
  induced alternating nonsingular form, and consider the flag bundle $\FF' :=
  \Fl^\iso_{n_1,\dots,n_{r-1}}(\EE'/\EE'^\perp)$.  There is a closed immersion
  $\lambda : \FF' \hookrightarrow \FF$ defined by taking $\lambda^* \calF(\FF)$
  to be the preimage of $\calF(\FF')$ in $\pi_{\FF'}^* \EE$, with $\pi_{\FF'}^*
  \EE$ and $0$ appended.

  We interpolate further between $\FF'$ and $\FF$.  Pick a complete flag of
  trivial totally coisotropic subbundles $\EE = \EE(0) \supseteq \EE(1)
  \supseteq \dots \supseteq \EE(n_r) = \EE'$ with successive quotients of rank
  $1$, and introduce the flag bundles $\FF(i) :=
  \Fl^\iso_{n_1,\dots,n_{r-1},n_r-i}(\EE(i)/\EE(i)^\perp)$.  There are natural
  closed immersions $\FF' \cong \FF(n_r) \buildrel {\lambda_{n_r-1}} \over
  \hookrightarrow \dots \buildrel {\lambda_1} \over \hookrightarrow \FF(1)
  \buildrel {\lambda_0} \over \hookrightarrow \FF(0) = \FF$ over $X$ whose
  composite is $\lambda$.

  The image of the closed immersion $\lambda_i : \FF(i+1) \hookrightarrow
  \FF(i)$ is defined by the vanishing of the composite map $ \pi_{\FF(i)}^*
  (\EE(i+1)^\perp/\EE(i)^\perp) \hookrightarrow \pi_{\FF(i)}^*
  (\EE(i)/\EE(i)^\perp) = \calF^r(\FF(i)) \twoheadrightarrow
  \calF^r_{r-1}(\FF(i))$, because on that zero locus, we have the containment
  $\pi_{\FF(i)}^* (\EE(i+1)^\perp/\EE(i)^\perp) \subseteq \calF_{r-1}(\FF(i))$,
  which implies also that $\calF^{r-1}(\FF(i)) \subseteq \pi_{\FF(i)}^*
  (\EE(i+1)/\EE(i)^\perp) $.  The source of the composite map is a trivial line
  bundle, and the codimension of $\lambda_i$ is the rank of
  $\calF_r(\FF(i))/\calF_{r-1}(\FF(i))^\perp$, for otherwise we wouldn't get a
  cycle class of dimension $0$ at the end of the proof, so
  \begin{align*}
    \lambda_{i*}( 1 ) & = c_\top(\calF^r_{r-1}(\FF(i))) \\ 
    & = c_\top(\calF^{r,r-1}(\FF(i))) \cdot c_\top(\calF^{r-1}_{r-1}(\FF(i))) \\
    & = c_{n_r-i}(\calF^{r,r-1}(\FF(i))) \cdot c_\top(\calF^{r-1}_{r-1}(\FF(i))) \\
    & = \lambda_{i-1}^* \dots \lambda_0^*  \left ( c_{n_r-i}(\calF^{r,r-1}(\FF)) \cdot c_\top( \calF^{r-1}_{r-1}(\FF)) \right ) \\
    & = \lambda_{i-1}^* \dots \lambda_0^*  \left ( c_{n_r-i}(\calF^{r,r-1}) \cdot c_\top( \calF^{r-1}_{r-1}) \right )
  \end{align*}
  Composing and using the product rule and Lemma \ref{lemma:JacobiTrudi} we get 
  \begin{align*}
    & \phantom{{}={}}
    \lambda_*( 1 ) = \lambda_{0*} \lambda_{1*} \dots \lambda_{n_r-1\,*} ( 1 ) \\
    & = c_{n_r}(\calF^{r,r-1}) \cdot \ldots \cdot c_{1}(\calF^{r,r-1}) \cdot  c_\top(\calF^{r-1}_{r-1})^{n_r} \\
    & = c_{n_r}(\calF^{r,r-1}) \cdot \ldots \cdot c_{1}(\calF^{r,r-1}) \cdot (-1)^{n_r(n-n_r)} \left ( c_{n_r}(\calF^{r,r-1}) \cdot c_{n_r}(\calF_{r-1,r})) \right ) ^{n - n_r} \\
    & = c_{n_r}(\calF^{r,r-1}) \cdot \ldots \cdot c_{1}(\calF^{r,r-1}) \cdot (-1)^{n_r(n-n_r)} \left ( (-1)^{n_r} c_{n_r}(\calF^{r,r-1})^2) \right ) ^{n - n_r} \\
    & = c_{1}(\calF^{r,r-1}) \cdot \ldots \cdot c_{n_r-1}(\calF^{r,r-1}) \cdot c_{n_r}(\calF^{r,r-1}) ^{2n - 2n_r + 1}
  \end{align*}
  
  One proves that $\eta$ is the class of a section by induction, where when $r$
  is replaced by $r-1$, $n$ is replaced by $n-n_r$ and $\FF$ is replaced by
  $\FF'$.  That finishes the proof.
\end{proof}

\begin{example}
  Let $H$ denote a hyperbolic
  plane over a field $F$, i.e., a vector space $F x \oplus F y$ of dimension
  $2$ with alternating form defined by $\langle x, y \rangle = 1$.

  Let $\pi : \FF = \Fl^\iso_3(H^3) \longrightarrow \Spec F$ be the isotropic
  Grassmannian of isotropic 3-planes in the six dimensional vector space $H^3$.
  Points of $\FF$ may be identified with isotropic 3-planes in $H^3$.
  The variety $\FF$ has dimension 6, and it has degree 16 in its Pl\"ucker embedding.

  The variety $\FF$ comes equipped with an exact sequence of vector bundles
  \[\xymatrix{
    0 \ar[r] & Q^* \ar[r] & \pi^*(H^3) \ar[r] & Q \ar[r] & 0,
   }
  \]
  where the rank 3 subbundle $Q^*$ is totally isotropic, and $Q$ is totally coisotropic.
  Let $c_i := c_i(Q)$, for $i=1,2,3$.  

  The intersection ring $A(\FF)$ of $\FF$ is
  \[ A(\FF) = \ZZ[c_1, c_2, c_3]/(c_1^2-2c_2, c_2^2-2c_1c_3, c_3^2).
  \]
  The given relations form a \Gb{} with respect to the graded reverse
  lexicographic monomial ordering.  The class of a point is represented by the monomial $c_1 c_2 c_3$.

  Let $W_i$ be a fixed totally isotropic subspace of $H^3$ of dimension $i$, for $i=1,2,3.$
  The class $c_1$ is the class of the subvariety 
    \[\sigma_1(W_3) := \{ V \in \FF \mid
    V \cap W_3 \neq 0\}.
    \]
  The subvariety $\sigma_1(W_3)$ has codimension 1 and degree 16, and is the
  hyperplane section of $\FF$ in the Pl\"ucker embedding ($c_1^6 = 16 c_1 c_2 c_3$).

  The class $c_2$ is the class of
    \[\sigma_2(W_2) := \{ V \in \FF \mid
    V \cap W_2 \neq 0\}.
    \]
  This codimension 2 subvariety has degree 8 in the Pl\"ucker embedding (%
  $c_2 c_1^4 = 4 c_2^3 = 8 c_1 c_2 c_3$).

  The class $c_3$ is the class of a zero section of $Q$:
    \[\sigma_3(W_1) := \{ V \in \FF \mid
    W_1 \subset V \}.
    \]
  This codimension 3 subvariety has degree 2 in the Pl\"ucker embedding ($c_1^3 c_3 = 2 c_1 c_2 c_3$).
\end{example}

\begin{example}
  We present some sample computations made by {\em Macaulay2}, using the
  algorithmic methods developed in this section.  As in the previous example, 
  let $H$ denote a hyperbolic
  plane over a field $F$, i.e., a vector space $F x \oplus F y$ of dimension
  $2$ with alternating form defined by $\langle x, y \rangle = 1$.
  \begin{align*}
    \int c_1 ( \calF^{1,0} ( \Fl^\iso_2(H^2))) ^ 3 & = 2 \\
    \int c_1 ( \calF^{1,0} ( \Fl^\iso_2(H^3))) ^ 7 & = 14 \\
    \int c_1 ( \calF^{1,0} ( \Fl^\iso_2(H^4))) ^ {11} & = 132 \\
    \int c_1 ( \calF^{1,0} ( \Fl^\iso_2(H^5))) ^ {15} & = 1430 \\
    \int c_1 ( \calF^{1,0} ( \Fl^\iso_3(H^3))) ^ 6 & = 16 \\
    \int c_1 ( \calF^{1,0} ( \Fl^\iso_4(H^4))) ^ {10} & = 768 \\
    \int c_1 ( \calF^{1,0} ( \Fl^\iso_5(H^5))) ^ {15} & = 292864
  \end{align*}
\end{example}

 \bibliographystyle{plain}
\bibliography{papers.bib}

\Addresses
\end{document}